\documentclass[bj,authoryear,noshowframe]{imsart}

\RequirePackage{amsmath,amssymb,amsthm,mathtools}
\RequirePackage{booktabs}
\RequirePackage{microtype}
\RequirePackage{array}

\makeatletter
\def\journal@name{Preprint}
\def\journal@url{https://arxiv.org/}
\makeatother

\startlocaldefs
\mathtoolsset{showonlyrefs=true}
\numberwithin{equation}{section}

\theoremstyle{plain}
\newtheorem{theorem}{Theorem}[section]
\newtheorem{proposition}[theorem]{Proposition}
\newtheorem{lemma}[theorem]{Lemma}
\newtheorem{corollary}[theorem]{Corollary}

\theoremstyle{definition}
\newtheorem{assumption}[theorem]{Assumption}

\theoremstyle{definition}
\newtheorem{remark}[theorem]{Remark}

\newcommand{\Smat}{\mathbb{S}}
\newcommand{\Splus}{\mathbb{S}_{+}}
\newcommand{\Id}{\operatorname{Id}}
\newcommand{\rank}{\operatorname{rank}}
\newcommand{\dist}{\operatorname{dist}}
\newcommand{\Gr}{\operatorname{Gr}}
\newcommand{\tr}{\operatorname{tr}}

\newcommand{\diag}{\operatorname{diag}}
\newcommand{\argmin}{\operatorname*{argmin}}

\newcommand{\qf}{\operatorname{qf}}
\endlocaldefs

\begin{document}

\begin{frontmatter}

\title{Nonstandard likelihood-ratio limits under semidefinite rank constraints}
\runtitle{Nonstandard likelihood-ratio limits}

\begin{aug}
\author[A]{\inits{D.}\fnms{Didier}~\snm{Concordet}\ead[label=e1]{didier.concordet@envt.fr}}
\orcid{0000-0003-3916-577X}
\address[A]{INTHERES, Universit\'e de Toulouse, INRAE, ENVT,
23 chemin des Capelles, 31076 Toulouse Cedex, France.
ORCID: \href{https://orcid.org/0000-0003-3916-577X}{0000-0003-3916-577X}
\printead[presep={,\ }]{e1}}
\end{aug}

\begin{abstract}
	We study likelihood-ratio tests for the hypothesis that a
	positive-semidefinite matrix has rank at most a prescribed value. The null
	hypothesis is stratified: points of maximal allowed rank lie on a regular
	boundary stratum, whereas lower-rank points are singular. Consequently, the
	usual chi-bar-square calibration on the top stratum does not by itself
	describe the whole composite null, especially along sequences whose rank
	changes at the local $n^{-1/2}$ scale.
	
	After profiling regular nuisance parameters, we derive a common reduced
	Gaussian experiment for every fixed null rank and for all admissible local
	rank transitions. On the top stratum, the classical chi-bar-square law is
	recovered. At lower ranks, the limit generally involves projection onto a
	nonconvex rank-constrained semidefinite set.
	
	Our main calibration result shows that, under isotropy, the top-stratum law
	is least favourable over all fixed null strata and all local null rank
	transitions. We also prove the corresponding transition dominance under
	arbitrary anisotropy when the active corank is one. Finally, on the top
	stratum, we obtain a conditional shape derivative for the limiting
	distribution and its critical value. Gaussian covariance models and
	finite-sample experiments illustrate nuisance profiling, rank transitions,
	anisotropy, and orientation sensitivity.
\end{abstract}


\end{frontmatter}

\section{Introduction}
\label{sec:introduction}

At regular interior points, Wilks' theorem gives a universal chi-square
calibration for likelihood-ratio tests.  At a boundary point or a singular
point of the null set, the local parameter space is no longer a linear
subspace, and the limiting statistic is governed by tangent geometry.  This
principle originates with \citet{Chernoff1954} and was developed in broad
constrained-inference settings by \citet{Shapiro1985},
\citet{SelfLiang1987}, and \citet{Geyer1994}; see also
\citet{SilvapulleSen2005}.  When the relevant tangent set is convex, the
limit is typically a chi-bar-square distribution.  Singular semi-algebraic
hypotheses can instead generate nonconvex tangent sets and limiting laws
outside that family \citep{Drton2009}.

Most likelihood-ratio calculations at semidefinite boundaries concern
convex local problems, notably tests of variance components
\citep{BaeyCournedeKuhn2019}.  Those results describe the regular boundary
stratum of the rank problem considered here, but not the singular strata
contained in a composite bounded-rank null.  More generally, likelihood-ratio
approximations can change rapidly near boundaries and singularities, so a
pointwise limit at the singular point need not describe nearby sequences well
\citep{MitchellAllmanRhodes2018}.  This distinction is central when the
singularity is an interface between rank strata.

We study the positive-semidefinite rank hypothesis
\begin{equation}
H_0:\ \Sigma\succeq0,\quad \rank(\Sigma)\le r,
\qquad
H_1:\ \Sigma\succeq0,
\label{eq:intro-hypotheses}
\end{equation}
where $\Sigma$ is a real symmetric $q\times q$ matrix and $0\le r<q$.  At a
null point of rank exactly $r$, the bounded-rank set is locally a smooth
stratum and classical conic likelihood-ratio theory gives a chi-bar-square
limit.  The null also contains matrices of rank $s<r$.  At these points the
local null permits the creation of positive eigenvalues only up to the
remaining rank budget $r-s$, and the resulting tangent set is nonconvex.

The statistical importance of lower-rank strata is already visible in the
work of \citet{ChenFang2019} on general matrix-rank inference.  For the
composite hypothesis $\rank(\Pi)\le r$, they show that calibration as though
the rank were always exactly $r$ can be invalid and that the top-rank limiting
law need not stochastically dominate the lower-rank laws.  They also study
local perturbations and construct bootstrap procedures for unrestricted
rectangular matrices.  The present problem is different: it concerns the
likelihood ratio under a positive-semidefinite constraint, and this additional
geometry both changes the lower-stratum limit and makes possible
least-favourable dominance results that fail in the unrestricted rank
problem.

The mathematical ingredients needed to expose that geometry are available
from variational analysis and convex geometry.  Tangent-cone formulas for
bounded-rank matrix sets have been developed for determinantal varieties and,
more recently, for symmetric and positive-semidefinite rank constraints
\citep{SchneiderUschmajew2015,OlikierMlinaricUschmajew2026,YangGaoYuan2025}.
Moreau decomposition and conic intrinsic volumes provide the classical
projection and chi-bar-square tools \citep{Moreau1962,McCoyTropp2014}.
Likewise, perturbation of optimization values and differentiation of moving
probability contents have established foundations
\citep{BonnansShapiro2000,Uryasev1994,HantouteHenrionPerezAros2017}.  We do
not claim these geometric or variational ingredients as new.  The statistical
question is how they combine, after nuisance profiling, to determine
likelihood-ratio calibration on an entire stratified semidefinite null and
along sequences that move between its strata.

The first contribution is an explicit active-block reduction of the general
Chernoff--Shapiro tangent-set limit.  After profiling regular nuisance
parameters and all matrix directions common to the null and alternative, the
limit depends only on the block acting on the kernel of the true matrix.  A
single reduced Gaussian experiment then describes every fixed rank stratum.
The same experiment, with a deterministic Gaussian shift, describes null
paths that cross a rank interface at the $n^{-1/2}$ scale.  Thus the reduction
identifies precisely where the classical chi-bar-square description survives
and where a lower-stratum nonconvex distance functional replaces it.

The second contribution concerns calibration over the composite null.  Under
isotropy, eigenvalue interlacing and compression yield a pathwise ordering:
the top-stratum statistic dominates every fixed lower stratum and every
admissible local null transition.  Consequently, the top-stratum critical
value is least favourable over this enlarged local null family.  Under
arbitrary anisotropy, we prove the corresponding transition dominance when
the active corank $q-r$ equals one.  We do not establish a general
least-favourable theorem for anisotropic active corank at least two; this is
the main unresolved interface problem.

The third contribution is a sensitivity analysis on the top stratum.  We
first specialize a standard envelope argument to the projection value for a
linearly transformed semidefinite cone.  Under explicit regular-level and
Gaussian-tail conditions, we then obtain a positive-level Hadamard derivative
for the distribution and critical value.  The result is deliberately
conditional: it is not asserted that every linearly transformed semidefinite
cone satisfies the required level-set continuity.  In active dimension two,
explicit chi-bar-square weights give an independent verification of the
quantile derivative.

Gaussian covariance models illustrate the statistical content of the
reduction.  One example shows how profiling an unknown residual variance
changes the active covariance and the chi-bar-square weights.  Further
experiments compare exact finite-sample likelihood ratios with lower-stratum
and rank-transition limits, display an anisotropic nonconvex lower-stratum
law, and verify the orientation derivative and a one-dimensional Armijo
ascent.

Section~\ref{sec:framework} introduces the statistical framework and the
geometric notation.  Section~\ref{sec:stratified} derives the stratified
limits and least-favourable comparisons.  Section~\ref{sec:sensitivity}
develops positive-level sensitivity on the top stratum.
Section~\ref{sec:applications} gives model calculations and numerical
validations, and Section~\ref{sec:discussion} discusses the scope of the
results and the remaining anisotropic interface problem.
Detailed proofs and computational formulas are collected in the
supplementary material following the references; each expanded proof is
identified locally in the main text.

\section{Framework and notation}
\label{sec:framework}

\subsection{Matrix spaces and elementary geometry}
\label{subsec:matrix-geometry}

All spaces are finite dimensional.  Let
\begin{equation}
\mathcal E=\mathbb R^{d_\psi}\oplus\Smat^q,
\qquad
\theta=(\psi,\Sigma),
\label{eq:framework-parameter-space}
\end{equation}
where $\psi\in\mathbb R^{d_\psi}$ is a regular nuisance component and
$\Sigma\in\Smat^q$ is a real symmetric matrix.  For integers $d\ge1$,
$\Smat^d$ denotes the space of real symmetric $d\times d$ matrices,
$\Splus^d$ its positive-semidefinite cone, and $\Smat_{++}^d$ its
positive-definite cone.  On $\Smat^d$ we use the Frobenius inner product and
norm
\begin{equation}
\langle A,B\rangle_F=\tr(AB),
\qquad
\|A\|_F=\langle A,A\rangle_F^{1/2}.
\label{eq:main-Frobenius-conventions}
\end{equation}
The Euclidean--Frobenius product on $\mathcal E$ is
\begin{equation}
\langle(a,A),(b,B)\rangle_{\mathcal E}
=a^\top b+\langle A,B\rangle_F,
\qquad
\|(a,A)\|_{\mathcal E}^2=\|a\|_2^2+\|A\|_F^2.
\label{eq:framework-product-inner-product}
\end{equation}
Direct sums appearing below carry the analogous product inner product.
For a linear operator $\mathcal L$ between Euclidean spaces,
$\mathcal L^*$ denotes its adjoint and $\|\mathcal L\|_{\mathrm{op}}$ its
operator norm.  If $\mathcal J$ is self-adjoint and positive definite, then
\begin{equation}
\|x\|_{\mathcal J}^2=\langle x,\mathcal Jx\rangle.
\label{eq:main-information-norm}
\end{equation}
For such an operator $\mathcal S$, $\mathcal S^{1/2}$ is its unique
self-adjoint positive-definite square root and
$\mathcal S^{-1/2}=(\mathcal S^{1/2})^{-1}$.  The notation
$\mathcal N_E(m,\mathcal V)$ denotes the Gaussian law on a Euclidean space
$E$ with mean $m$ and covariance operator $\mathcal V$; $\Id$ is the
identity operator on the relevant space and $I_d$ the $d\times d$ identity
matrix.  A standard Gaussian matrix $Y\in\Smat^d$ has independent
$N(0,1)$ coordinates in the Frobenius-orthonormal basis $E_{ii}$ and
$F_{ij}=(E_{ij}+E_{ji})/\sqrt2$, $i<j$, where $E_{ij}$ is the matrix unit.
Thus the off-diagonal matrix entries of $Y$ have variance $1/2$.
For a nonempty closed subset $A$ of a Euclidean space with norm
$\|\cdot\|$, define
\begin{equation}
	\dist(x,A)=\inf_{z\in A}\|x-z\|.
	\label{eq:main-distance-definition}
\end{equation}
When the ambient space is a matrix space equipped with the Frobenius norm,
we write
\begin{equation}
	\dist_F(X,A)
	=
	\inf_{Z\in A}\|X-Z\|_F.
\end{equation} If $A$ is closed and convex,
$\Pi_A(x)$ is its unique metric projection.  For a convex cone $C$,
\begin{equation}
C^\circ=\{y:\langle y,z\rangle\le0\ \text{for every }z\in C\}
\label{eq:main-polar-definition}
\end{equation}
is its polar cone; $\partial C$ and $\operatorname{int}(C)$ denote its
boundary and interior.  The Bouligand tangent cone to a set $A$ at $x\in A$
is
\begin{equation}
T_A(x)=
\left\{h:\exists\ t_n\downarrow0,\ x_n\in A
\text{ such that }(x_n-x)/t_n\to h\right\}.
\label{eq:main-Bouligand-definition}
\end{equation}
For nonempty compact sets $A,B$, the Hausdorff distance is
\begin{equation}
d_{\mathrm H}(A,B)=
\max\left\{
\sup_{a\in A}\dist(a,B),
\sup_{b\in B}\dist(b,A)
\right\}.
\label{eq:main-Hausdorff-definition}
\end{equation}
Convergence on bounded sets means Hausdorff convergence after intersection
with every closed bounded ball.

We order eigenvalues as
$\lambda_1(M)\ge\cdots\ge\lambda_d(M)$ and write $x_+=\max(x,0)$.
The orthogonal group is
$\operatorname O(d)=\{O\in\mathbb R^{d\times d}:O^\top O=I_d\}$.
For random variables, $X_n\Rightarrow X$ and
$X_n\xrightarrow{\Pr}X$ denote convergence in distribution and in
probability, and $X\stackrel d=Y$ denotes equality in distribution.  We write
$X\preceq_{\mathrm{st}}Y$ when $\Pr(X>t)\le\Pr(Y>t)$ for every real $t$.
For a distribution function $F$ and $\alpha\in(0,1)$,
\begin{equation}
q_\alpha(F)=F^{-1}(1-\alpha)
:=\inf\{c:F(c)\ge1-\alpha\}
\label{eq:main-quantile-convention}
\end{equation}
is its upper-$\alpha$ critical value.  The symbol $\mathbf1_A$ denotes the
indicator of an event or set $A$.

\subsection{Statistical experiment and rank constraint}
\label{subsec:statistical-framework}

The model is assumed to admit a regular local extension to an open
neighbourhood of the true parameter
$\theta_0=(\psi_0,\Sigma_0)$.  The null and alternative parameter sets are
$\Theta_0\subseteq\Theta_1$, with the restrictions on the matrix component
given by \eqref{eq:intro-hypotheses}.  We write $\Pr_\theta$ for the data law
at parameter $\theta$ and $\ell_n(\theta)$ for the sample log-likelihood.
The likelihood-ratio statistic is
\begin{equation}
\Lambda_n
=2\left\{
\sup_{\theta\in\Theta_1}\ell_n(\theta)
-
\sup_{\theta\in\Theta_0}\ell_n(\theta)
\right\}.
\label{eq:LRT}
\end{equation}
Throughout Sections~\ref{sec:framework} and~\ref{sec:stratified},
$\Sigma_0\succeq0$ has rank $s\le r$.  Its kernel, kernel dimension, and
orthogonal projector are
\begin{equation}
K_s=\ker(\Sigma_0),
\qquad
k_s=q-s,
\qquad
P_s:\mathbb R^q\to K_s.
\label{eq:framework-kernel-notation}
\end{equation}
Choose an orthonormal basis matrix $U_s\in\mathbb R^{q\times k_s}$ for
$K_s$.  The active compression and its adjoint are
\begin{equation}
\mathcal A_{P_s}(H)=U_s^\top H U_s,
\qquad
\mathcal A_{P_s}^*(B)=U_sBU_s^\top.
\label{eq:framework-active-compression}
\end{equation}
The number of additional positive eigenvalues allowed locally is
$m_s=r-s$, and for $0\le m\le k$ we set
\begin{equation}
\mathcal K_m^k
=\{B\in\Splus^k:\rank(B)\le m\}.
\label{eq:framework-rank-cone}
\end{equation}
These definitions isolate the block on which the null and alternative tangent
sets can differ.

\subsection{Uniform LAN and nuisance profiling}
\label{subsec:LAN-profile}

\begin{assumption}[Uniform LAN]
\label{ass:LAN}
There exist random elements $Z_n\in\mathcal E$ and a self-adjoint
positive-definite operator $\mathcal I$ such that, for every compact
$C\subset\mathcal E$,
\begin{equation}
\sup_{h\in C}
\left|
\ell_n(\theta_0+n^{-1/2}h)-\ell_n(\theta_0)
-\langle h,Z_n\rangle
+\frac12\langle h,\mathcal Ih\rangle
\right|
\xrightarrow{\Pr}0,
\label{eq:LAN}
\end{equation}
and $Z_n\Rightarrow Z\sim\mathcal N_{\mathcal E}(0,\mathcal I)$, meaning
that $\mathbb EZ=0$ and
$\operatorname{Cov}\{\langle u,Z\rangle,\langle v,Z\rangle\}
=\langle u,\mathcal Iv\rangle$ for $u,v\in\mathcal E$.
\end{assumption}

Partition $Z=(Z_\psi,Z_\Sigma)$ and $\mathcal I$ according to
$\mathcal E=\mathbb R^{d_\psi}\oplus\Smat^q$, and define
\begin{equation}
Z_\Sigma^{\mathrm{eff}}
=Z_\Sigma-
\mathcal I_{\Sigma\psi}\mathcal I_{\psi\psi}^{-1}Z_\psi,
\qquad
\mathcal I_{\mathrm{eff}}
=\mathcal I_{\Sigma\Sigma}
-
\mathcal I_{\Sigma\psi}\mathcal I_{\psi\psi}^{-1}
\mathcal I_{\psi\Sigma}.
\label{eq:efficient-information}
\end{equation}
Then $\mathcal I_{\mathrm{eff}}$ is positive definite and
$Z_\Sigma^{\mathrm{eff}}\sim
\mathcal N_{\Smat^q}(0,\mathcal I_{\mathrm{eff}})$.  With
\begin{equation}
G=\mathcal I_{\mathrm{eff}}^{-1}Z_\Sigma^{\mathrm{eff}},
\label{eq:G}
\end{equation}
profiling the nuisance direction reduces the Gaussian criterion, up to an
additive random constant, to
$-\frac12\|H-G\|_{\mathcal I_{\mathrm{eff}}}^2$.  This is the standard
Schur-complement calculation; Lemma~\ref{lem:supp-schur-profile} gives the
complete calculation and makes the dependence on the nuisance direction
explicit.

\section{Stratified active-block likelihood-ratio limits}
\label{sec:stratified}

\subsection{Tangent geometry on every rank stratum}

With the kernel notation and active compression introduced in
\eqref{eq:framework-kernel-notation}--\eqref{eq:framework-active-compression},
the tangent cone to the unrestricted PSD alternative
is classical:
\begin{equation}
T_1(P_s)
=
\{H\in\Smat^q:\mathcal A_{P_s}(H)\succeq0\}.
\label{eq:alternative-tangent}
\end{equation}
For the bounded-rank null, the lower strata require an additional rank
restriction.  The underlying tangent geometry is known in variational
analysis.  Formulas for the real determinantal variety were used by
\citet{SchneiderUschmajew2015}; \citet{OlikierMlinaricUschmajew2026}
identified a gap in that proof and supplied complete alternatives, while
\citet{YangGaoYuan2025} gives a unified treatment that includes symmetric and
positive-semidefinite bounded-rank sets.  We state the positive-semidefinite
specialization in the active-block form needed for the likelihood-ratio
reduction and include a self-contained proof based on a Schur complement and
an explicit factor path.

\begin{theorem}[Tangent cone to the bounded-rank PSD null]
\label{thm:null-tangent}
Let $\Sigma_0\succeq0$ have rank $s\le r$.  Then the Bouligand tangent cone
to
$\{\Sigma\succeq0:\rank(\Sigma)\le r\}$ at $\Sigma_0$ is
\begin{equation}
T_{0,s}^{(r)}(P_s)
=
\left\{
H\in\Smat^q:
\mathcal A_{P_s}(H)\succeq0,
\ \rank\{\mathcal A_{P_s}(H)\}\le r-s
\right\}.
\label{eq:null-tangent}
\end{equation}
Equivalently,
\begin{equation}
T_{0,s}^{(r)}(P_s)
=
\{H:\mathcal A_{P_s}(H)\in\mathcal K_{m_s}^{k_s}\},
\qquad m_s=r-s.
\label{eq:null-tangent-short}
\end{equation}
For $s=r$, this reduces to the tangent space
$\{H:\mathcal A_{P_r}(H)=0\}$ of the rank-$r$ stratum.  For $s<r$ the
null tangent cone is nonconvex.
\end{theorem}

\begin{proof}
Choose an orthogonal basis adapted to
$\operatorname{range}(\Sigma_0)\oplus K_s$, so that
\begin{equation}
\Sigma_0=
\begin{pmatrix}
\Lambda&0\\0&0
\end{pmatrix},
\qquad
\Lambda\in\Smat_{++}^s,
\qquad
H=
\begin{pmatrix}
A&B\\B^{\top}&C
\end{pmatrix}.
\label{eq:block-tangent-proof}
\end{equation}
We first prove necessity.  Let $t_n\downarrow0$ and
$H_n\to H$ be such that
$\Sigma_n=\Sigma_0+t_nH_n\succeq0$ and
$\rank(\Sigma_n)\le r$.  For $n$ large, the upper-left block
$\Lambda+t_nA_n$ is positive definite.  The Schur complement
\begin{equation}
R_n
=
t_nC_n-t_n^2B_n^{\top}
(\Lambda+t_nA_n)^{-1}B_n
\label{eq:Schur-tangent}
\end{equation}
is positive semidefinite and satisfies
\begin{equation}
\rank(R_n)
=
\rank(\Sigma_n)-s
\le r-s.
\end{equation}
Dividing \eqref{eq:Schur-tangent} by $t_n$ gives
\begin{equation}
C_n-t_nB_n^{\top}
(\Lambda+t_nA_n)^{-1}B_n
\longrightarrow C.
\end{equation}
The set of positive-semidefinite matrices of rank at most $r-s$ is closed;
hence $C\succeq0$ and $\rank(C)\le r-s$.

For sufficiency, suppose that $C\succeq0$ and
$\rank(C)=m\le r-s$.  Write $C=WW^{\top}$ with
$W\in\mathbb R^{k_s\times m}$.  Choose
$E\in\mathbb R^{s\times s}$ solving
\begin{equation}
\Lambda^{1/2}E^{\top}+E\Lambda^{1/2}=A,
\label{eq:E-equation}
\end{equation}
and set $F=B^{\top}\Lambda^{-1/2}$.  Define the factor matrix
\begin{equation}
X(t)
=
\left[
\begin{array}{cc}
\Lambda^{1/2}+tE & 0\\
tF & \sqrt t\,W
\end{array}
\right].
\label{eq:factor-path}
\end{equation}
Then $\Sigma(t)=X(t)X(t)^{\top}$ is positive semidefinite and
$\rank\{\Sigma(t)\}\le s+m\le r$.  Expanding
\eqref{eq:factor-path} gives
\begin{equation}
\Sigma(t)
=
\Sigma_0+t
\begin{pmatrix}
A&B\\B^{\top}&C
\end{pmatrix}
+O(t^2).
\end{equation}
Thus $H$ belongs to the Bouligand tangent cone.  Since the lower-right
block is precisely $\mathcal A_{P_s}(H)$, the claimed formula follows.
\end{proof}

The proof above is complete.  Theorem~\ref{thm:supp-null-tangent} in the
supplement restates the same tangent-cone identity with the block argument
expanded.  Proposition~\ref{prop:supp-local-Hausdorff} proves the additional,
strictly stronger uniform local-Hausdorff approximation used below to verify
the set-convergence part of Assumption~\ref{ass:local}; that uniform property
does not follow merely by restating the pointwise tangent-cone formula.

Theorem~\ref{thm:null-tangent} identifies the precise obstruction to a
single chi-bar-square calibration over the full null.  Only the top stratum
has a linear null tangent.  Lower strata compare a convex alternative cone
to a nested nonconvex cone.

\begin{assumption}[Local set convergence and localisation]
\label{ass:local}
The rescaled null and alternative parameter sets converge on bounded sets to
$\mathbb R^{d_\psi}\oplus T_{0,s}^{(r)}(P_s)$ and
$\mathbb R^{d_\psi}\oplus T_1(P_s)$, respectively, and the corresponding
local constrained maximisers are stochastically bounded.
\end{assumption}

For the direct semidefinite rank constraint, the required local Hausdorff
convergence is proved in Proposition~\ref{prop:supp-local-Hausdorff} of the supplementary material by a uniform version of the factor construction in
Theorem~\ref{thm:null-tangent}.  This is consistent with the general
Chernoff regularity of semi-algebraic hypotheses established by
\citet{Drton2009}.  We retain the assumption in the main theorem to permit
smooth model embeddings and to keep probabilistic and geometric conditions
separate.

\subsection{Stratified canonical reduction}

The adjoint of $\mathcal A_{P_s}$ is
$\mathcal A_{P_s}^{*}(B)=U_sBU_s^{\top}$.  Define the reduced active
covariance
\begin{equation}
\mathcal S_{P_s}
=
\mathcal A_{P_s}\mathcal I_{\mathrm{eff}}^{-1}
\mathcal A_{P_s}^{*}.
\label{eq:SPs}
\end{equation}
It is self-adjoint and positive definite on $\Smat^{k_s}$.  If
$B_s=\mathcal A_{P_s}(G)$, then
$B_s\sim\mathcal N_{\Smat^{k_s}}(0,\mathcal S_{P_s})$ and
\begin{equation}
\inf_{\mathcal A_{P_s}(H)=B}
\|H-G\|_{\mathcal I_{\mathrm{eff}}}^2
=
\left\langle
B-B_s,\mathcal S_{P_s}^{-1}(B-B_s)
\right\rangle.
\label{eq:quotient}
\end{equation}
This quotient identity profiles every off-kernel block and is the same on
all strata.

For the likelihood-ratio statistic $\Lambda_n$ defined in
\eqref{eq:LRT}, define
\begin{equation}
\mathcal C_{P_s}
=
\mathcal S_{P_s}^{-1/2}(\Splus^{k_s}),
\qquad
\mathcal D_{P_s,m_s}
=
\mathcal S_{P_s}^{-1/2}(\mathcal K_{m_s}^{k_s}).
\label{eq:CD}
\end{equation}
The first set is a closed convex cone; the second is a closed, generally
nonconvex, subcone.

\begin{theorem}[Stratified active-block limit]
\label{thm:stratified-limit}
Under Assumptions~\ref{ass:LAN} and \ref{ass:local}, if
$\rank(\Sigma_0)=s\le r$, then
\begin{equation}
\Lambda_n
\Rightarrow
\Delta_{s,r}(P_s)
:=
\dist_F^2(Y,\mathcal D_{P_s,m_s})
-
\dist_F^2(Y,\mathcal C_{P_s}),
\qquad
Y\sim\mathcal N_{\Smat^{k_s}}(0,\Id).
\label{eq:stratified-limit}
\end{equation}
The law is independent of the selected orthonormal basis of $K_s$.
Here and below, a standard Gaussian element of $\Smat^{k_s}$ has
independent $N(0,1)$ coefficients in the Frobenius-orthonormal basis
$E_{ii}$ and $F_{ij}=(E_{ij}+E_{ji})/\sqrt2$; consequently, its matrix
off-diagonal entries have variance $1/2$.
\end{theorem}

\begin{proof}
By Assumption~\ref{ass:local}, the constrained local maximisers may be
restricted, with arbitrarily large probability, to a deterministic compact
ball.  On that ball the remainder in \eqref{eq:LAN} is uniformly negligible.
Local convergence of the rescaled parameter sets and strict concavity of the
Gaussian criterion therefore give
\begin{equation}
\Lambda_n
\Rightarrow
\inf_{H\in T_{0,s}^{(r)}(P_s)}
\|H-G\|_{\mathcal I_{\mathrm{eff}}}^2
-
\inf_{H\in T_1(P_s)}
\|H-G\|_{\mathcal I_{\mathrm{eff}}}^2.
\label{eq:pre-active-limit}
\end{equation}
The joint convergence of the two suprema is detailed in Proposition~\ref{prop:supp-supremum-limit} of the supplementary material.

For completeness, the affine profile underlying \eqref{eq:quotient} has
minimiser
\begin{equation}
H_B
=
G+
\mathcal I_{\mathrm{eff}}^{-1}
\mathcal A_{P_s}^{*}
\mathcal S_{P_s}^{-1}(B-B_s).
\label{eq:HB-main}
\end{equation}
Indeed, $\mathcal A_{P_s}(H_B)=B$, and
$H_B-G$ is orthogonal in the $\mathcal I_{\mathrm{eff}}$ metric to
$\ker(\mathcal A_{P_s})$.  Substitution gives \eqref{eq:quotient}.
Applying it to the two sets in Theorem~\ref{thm:null-tangent} transforms
\eqref{eq:pre-active-limit} into
\begin{equation}
\begin{split}
&\inf_{B\in\mathcal K_{m_s}^{k_s}}
\left\langle B-B_s,
\mathcal S_{P_s}^{-1}(B-B_s)\right\rangle\\
&\hspace{2cm}-
\inf_{B\in\Splus^{k_s}}
\left\langle B-B_s,
\mathcal S_{P_s}^{-1}(B-B_s)\right\rangle.
\end{split}
\label{eq:active-two-distances}
\end{equation}
Since $B_s$ has covariance $\mathcal S_{P_s}$,
$Y=\mathcal S_{P_s}^{-1/2}B_s$ is standard Gaussian.  The change of
variable $C=\mathcal S_{P_s}^{-1/2}B$ turns
\eqref{eq:active-two-distances} into \eqref{eq:stratified-limit}.

Finally, if $U_s$ is replaced by $U_sO$ with $O\in\operatorname O(k_s)$,
the compression is conjugated by
$\mathcal Q_O(B)=O^{\top}BO$.  Both sets in \eqref{eq:CD} are therefore
mapped by the same orthogonal transformation, and the standard Gaussian law
is invariant under $\mathcal Q_O$.  The limiting distribution is basis
independent.
\end{proof}

The display above gives the geometric part of the proof.  The detailed
probabilistic passage from uniform LAN to the joint pair of constrained
Gaussian suprema is Proposition~\ref{prop:supp-supremum-limit}; nuisance
profiling and the affine active-block profile are expanded in
Lemmas~\ref{lem:supp-schur-profile} and~\ref{lem:supp-affine-profile}.
Together these results supply the full proof of
Theorem~\ref{thm:stratified-limit}.

Theorem~\ref{thm:stratified-limit} is the central reduction.  It separates
the model-specific derivation of $\mathcal I_{\mathrm{eff}}$ and $P_s$ from
a universal active-space problem.  It also shows why the top and lower
strata must be distinguished.

\subsection{Local alternatives and null rank transitions}

The active representation also gives the local power experiment without any
additional model-specific calculation.  Consider the local sequence
\begin{equation}
\theta_n
=
\theta_0+n^{-1/2}(a,H),
\label{eq:local-alternative}
\end{equation}
inside the ambient regular model, where $a\in\mathbb R^{d_\psi}$ and
$H\in\Smat^q$ are fixed nuisance and matrix directions.  Under the LAN
assumption this is a contiguous sequence of experiments.  Le Cam's third lemma shifts the mean of
the efficient score by $\mathcal I_{\mathrm{eff}}H$; the nuisance component
$a$ cancels from this mean.  Hence $G$ in \eqref{eq:G} has limiting mean
$H$, the active block $B_s$ has mean $\mathcal A_{P_s}(H)$, and the whitened
mean is
\begin{equation}
\mu_s(H)
=
\mathcal S_{P_s}^{-1/2}\mathcal A_{P_s}(H).
\label{eq:local-mean}
\end{equation}

\begin{corollary}[Stratified local-alternative limit]
\label{cor:local-alternative}
Under \eqref{eq:local-alternative}, the limit in
Theorem~\ref{thm:stratified-limit} remains valid with
\begin{equation}
Y\sim\mathcal N_{\Smat^{k_s}}\{\mu_s(H),\Id\}.
\label{eq:noncentral-stratified}
\end{equation}
In particular, the asymptotic power against $H$ at critical value $c$ is
\begin{equation}
\Pr\left[
\dist_F^2\{Y,\mathcal D_{P_s,m_s}\}
-
\dist_F^2\{Y,\mathcal C_{P_s}\}>c
\right].
\label{eq:local-power}
\end{equation}
\end{corollary}

The score shift and nuisance cancellation used in this corollary are derived
in Subsection~\ref{subsec:supp-local-transitions} of the supplement.

On the top stratum this is a noncentral conic projection law.  On lower
strata it is a translated nonconvex distance-difference law.  Formula
\eqref{eq:local-mean} shows again that the original model enters local power
through the same reduced covariance and active compression as null
calibration.

A local alternative need not leave the null.  This observation turns the
preceding power calculation into a transition theorem at a rank interface.

\begin{corollary}[Null rank-transition limit]
\label{cor:null-transition}
Let $t\mapsto\theta(t)\in\Theta_0$ be a null path such that
\begin{equation}
\theta(t)=\theta_0+t(a,H)+o(t),
\qquad t\downarrow0,
\label{eq:null-transition-path}
\end{equation}
and put $\theta_n=\theta(n^{-1/2})$.  Under $\Pr_{\theta_n}$,
\begin{equation}
\Lambda_n\Rightarrow
\Delta_{s,r}^{\mathrm{tr}}(P_s;H)
:=
\dist_F^2\{Y+\mu_s(H),\mathcal D_{P_s,m_s}\}
-
\dist_F^2\{Y+\mu_s(H),\mathcal C_{P_s}\},
\label{eq:null-transition-limit}
\end{equation}
where $Y$ is standard Gaussian and $\mu_s(H)$ is given by
\eqref{eq:local-mean}.  Necessarily
$\mathcal A_{P_s}(H)\in\mathcal K_{m_s}^{k_s}$; conversely, every element of
$\mathcal K_{m_s}^{k_s}$ occurs as the active first-order increment of a
null path.
\end{corollary}

Subsection~\ref{subsec:supp-local-transitions} gives the detailed Le Cam
argument for this translated experiment and proves that the admissible active
drifts are exactly the elements of $\mathcal K_{m_s}^{k_s}$.

The last assertion follows from the factor path used in
Theorem~\ref{thm:null-tangent}.  Off-kernel first-order blocks may change the
path but disappear after active profiling.  Thus the transition family is
indexed canonically by the admissible active drift.

\subsection{Top-stratum chi-bar-square law}

If $s=r$, then $m_s=0$ and
$\mathcal D_{P_r,0}=\{0\}$.  Moreau's decomposition \citep{Moreau1962} gives
\begin{equation}
\Delta_{r,r}(P_r)
=
\|\Pi_{\mathcal C_{P_r}}(Y)\|_F^2.
\label{eq:top-projection}
\end{equation}
Let $d_r=(q-r)(q-r+1)/2$.

\begin{corollary}[Top-stratum chi-bar-square representation]
\label{cor:chibar}
At a rank-$r$ null point,
\begin{equation}
\Pr\{\Delta_{r,r}(P_r)\le c\}
=
\sum_{j=0}^{d_r}
v_j(\mathcal C_{P_r})F_{\chi_j^2}(c),
\qquad c\ge0,
\label{eq:chibar}
\end{equation}
where $F_{\chi_j^2}$ is the distribution function of a chi-square
variable with $j$ degrees of freedom, $\chi_0^2$ is the point mass at zero,
and $v_0(C),\ldots,v_{d_r}(C)$ are the nonnegative conic intrinsic-volume
weights of a closed convex cone $C$; they sum to one and determine its
chi-bar-square distribution.  This is the conic Steiner representation of
\citet{McCoyTropp2014}.
\end{corollary}

We next compare two reduced calibration problems on the same active space
$\Smat^k$.  Let $\mathcal S_1$ and $\mathcal S_2$ be self-adjoint
positive-definite operators, and for $i=1,2$ define
\begin{equation}
\mathcal C_i=\mathcal S_i^{-1/2}(\Splus^k),
\qquad
\mathcal D_{i,m}=\mathcal S_i^{-1/2}(\mathcal K_m^k).
\label{eq:two-reduced-calibrations}
\end{equation}
If, for some $a>0$ and orthogonal operator
$\mathcal Q:\Smat^k\to\Smat^k$,
\begin{equation}
\mathcal S_2
=
a\,\mathcal Q\mathcal S_1\mathcal Q^*,
\qquad
\mathcal Q(\Splus^k)=\Splus^k,
\qquad
\mathcal Q(\mathcal K_m^k)=\mathcal K_m^k,
\label{eq:sufficient-equivalence}
\end{equation}
then $\mathcal C_2=\mathcal Q\mathcal C_1$ and
$\mathcal D_{2,m}=\mathcal Q\mathcal D_{1,m}$: positive scalar dilations do
not change cones.  Hence the two Gaussian distance-difference statistics
have the same law.  This sufficient condition is not necessary; equality of
the limiting laws requires only equality in distribution of the corresponding
distance functionals.

\subsection{Continuity of stratified calibration}

Continuity is needed both for compact least-favourable problems and for
plug-in calibration.  Let $\mathfrak P_k$ denote the positive-definite
self-adjoint operators on $\Smat^k$.  For fixed $m\le k$, write
\begin{equation}
\Delta_m(\mathcal S,Y)
=
\dist_F^2\{Y,\mathcal S^{-1/2}(\mathcal K_m^k)\}
-
\dist_F^2\{Y,\mathcal S^{-1/2}(\Splus^k)\}.
\label{eq:generic-stratum-stat}
\end{equation}
Let $F_{m,\mathcal S}$ be the distribution function of
$\Delta_m(\mathcal S,Y)$ for standard Gaussian $Y$, and define
\begin{equation}
q_{\alpha,m}(\mathcal S)
:=q_\alpha(F_{m,\mathcal S})
=\inf\{c:F_{m,\mathcal S}(c)\ge1-\alpha\}.
\label{eq:generic-stratum-quantile}
\end{equation}

\begin{proposition}[Continuity of the stratified law]
\label{prop:law-continuity}
If $\mathcal S_n\to\mathcal S$ in operator norm with all operators positive
definite, then
\begin{equation}
\Delta_m(\mathcal S_n,Y)
\longrightarrow
\Delta_m(\mathcal S,Y)
\quad\hbox{for every }Y\in\Smat^k.
\label{eq:pointwise-continuity}
\end{equation}
Consequently the distribution functions converge at every continuity point
of the limiting law.  If the $(1-\alpha)$ quantile is positive and the
limiting distribution is strictly increasing there, then the corresponding
quantiles converge.
\end{proposition}

\begin{proof}
The inverse-square-root map is continuous on $\mathfrak P_k$.  On every
bounded ball, the images
$\mathcal S_n^{-1/2}(\mathcal K_m^k)$ and
$\mathcal S_n^{-1/2}(\Splus^k)$ converge in Hausdorff distance to their
limits.  Distance to a closed set is stable under such local Hausdorff
convergence, proving \eqref{eq:pointwise-continuity}.  The random variables
are bounded by a constant multiple of $1+\|Y\|_F^2$ when the operators range
in a compact subset of $\mathfrak P_k$, so dominated convergence gives weak
convergence.  Quantile convergence follows from strict increase at the
target level.
\end{proof}

\begin{corollary}[Plug-in critical values]
\label{cor:plugin-calibration}
Let $\widehat{\mathcal S}_n$ be a positive-definite estimator satisfying
$\widehat{\mathcal S}_n\to\mathcal S$ in probability.  Under the positive
and strictly increasing quantile condition in
Proposition~\ref{prop:law-continuity},
\begin{equation}
q_{\alpha,m}(\widehat{\mathcal S}_n)
\xrightarrow{\Pr}
q_{\alpha,m}(\mathcal S).
\label{eq:plugin-quantile}
\end{equation}
Consequently, whenever the likelihood-ratio statistic has the fixed-stratum
limit $\Delta_m(\mathcal S,Y)$, replacing the limiting critical value by the
plug-in value is asymptotically equivalent at continuity points of the law.
\end{corollary}

This result supplies the continuity needed for plug-in and compact
least-favourable constructions.  It does not imply differentiability, which
is substantially more delicate even on the top stratum.

\subsection{Lower strata and isotropic transition dominance}

For $s<r$, the null projection in \eqref{eq:stratified-limit} is a
rank-constrained approximation.  Under anisotropy it is a nonconvex
quadratic problem.  In the isotropic case both fixed-stratum and transition
laws become explicit.  For $C\in\mathcal K_m^k$, define
\begin{equation}
\Delta_m^{\mathrm{tr}}(C;Y)
=
\dist_F^2(Y+C,\mathcal K_m^k)
-
\dist_F^2(Y+C,\Splus^k).
\label{eq:isotropic-transition-stat}
\end{equation}
The case $C=0$ is the fixed-stratum law; an admissible $C\ne0$ is the
whitened drift in Corollary~\ref{cor:null-transition}.

\begin{proposition}[Isotropic spectral representation]
\label{prop:isotropic-spectral}
Suppose that $\mathcal S_{P_s}$ is proportional to the identity on
$\Smat^{k_s}$.  Let $C\in\mathcal K_{m_s}^{k_s}$ and let
$\lambda_1\ge\cdots\ge\lambda_{k_s}$ denote the eigenvalues of $Y+C$, where
$Y$ is standard Gaussian in $\Smat^{k_s}$.  Then
\begin{equation}
\Delta_{m_s}^{\mathrm{tr}}(C;Y)
=
\sum_{j=m_s+1}^{k_s}\{\lambda_j(Y+C)_+\}^2.
\label{eq:isotropic-spectral}
\end{equation}
\end{proposition}

\begin{proof}
Positive scaling of $\mathcal S_{P_s}$ does not change either cone.  Under
the Frobenius metric, projection onto $\Splus^{k_s}$ replaces negative
eigenvalues by zero.  Projection onto $\mathcal K_{m_s}^{k_s}$ additionally
retains only the $m_s$ largest positive eigenvalues.  Applying this spectral
truncation to $Y+C$ and subtracting the squared residuals gives
\eqref{eq:isotropic-spectral}.
\end{proof}

Subsection~\ref{subsec:supp-spectral-projection} gives the corresponding
spectral approximation argument in expanded form.

The number of retained tail terms is
$k_s-m_s=q-r$, independent of $s$.  Interlacing orders the central laws, and
a compression to the kernel of $C$ controls the whole transition family.

\begin{theorem}[Top-stratum dominance under isotropy]
\label{thm:isotropic-dominance}
Assume that the reduced covariance is proportional to the identity on
every stratum and set $p=q-r$.  In this theorem write
$\Delta_{s,r}:=\Delta_{s,r}(P_s)$; its law is independent of $P_s$ under
isotropy.  Let $Y_{k_s}$ denote the standard Gaussian matrix in
$\Smat^{k_s}$ used to construct that stratum's statistic.  The fixed-stratum
limits can be coupled so that
\begin{equation}
\Delta_{0,r}\le\Delta_{1,r}\le\cdots\le\Delta_{r,r}
\quad\hbox{almost surely}.
\label{eq:pathwise-dominance}
\end{equation}
Moreover, for every $s\le r$ and every
$C\in\mathcal K_{m_s}^{k_s}$, there is a coupling with a standard Gaussian
$G_p\in\Smat^p$ such that
\begin{equation}
\Delta_{m_s}^{\mathrm{tr}}(C;Y_{k_s})
\le
T_p(G_p)
:=
\sum_{j=1}^{p}\{\lambda_j(G_p)_+\}^2
\quad\hbox{almost surely}.
\label{eq:isotropic-transition-dominance}
\end{equation}
The variable $T_p(G_p)$ has the rank-$r$ top-stratum law.  Hence the
rank-$r$ critical value controls every fixed stratum and every local null
rank transition in the isotropic experiment.
\end{theorem}

\begin{proof}
For \eqref{eq:pathwise-dominance}, couple standard Gaussian matrices of
successive dimensions as principal submatrices and apply Cauchy interlacing,
as shown in Section~\ref{sec:supp-isotropic} of the supplementary material.

For \eqref{eq:isotropic-transition-dominance}, write $k=k_s$ and $m=m_s$.
Because $\rank(C)\le m$, choose a $p$-dimensional subspace
$L\subseteq\ker(C)$ and an orthonormal basis matrix $V\in\mathbb R^{k\times p}$.
The separation inequalities for a compression give
\begin{equation}
\lambda_{m+j}(Y+C)
\le
\lambda_j\{V^\top(Y+C)V\}
=
\lambda_j(V^\top YV),
\qquad j=1,\ldots,p.
\end{equation}
The compression $V^\top YV$ is standard Gaussian in $\Smat^p$.  Taking
positive parts, squaring, and summing proves the result.
\end{proof}

The fixed-stratum coupling and the transition compression are detailed,
respectively, in Subsections~\ref{subsec:supp-fixed-ordering} and
\ref{subsec:supp-isotropic-transition} of the supplement.

\begin{corollary}[Isotropic calibration along rank interfaces]
\label{cor:isotropic-transition-calibration}
Let $c_{\alpha,r}^{\mathrm{top}}$ be a positive upper $\alpha$ critical value
of the isotropic top-stratum law.  Along every null path covered by
Corollary~\ref{cor:null-transition},
\begin{equation}
\limsup_{n\to\infty}
\Pr_{\theta_n}\{\Lambda_n>c_{\alpha,r}^{\mathrm{top}}\}
\le\alpha.
\label{eq:isotropic-transition-size}
\end{equation}
The same conclusion holds uniformly over compact drift families whenever
the LAN and localisation assumptions hold uniformly over those families.
\end{corollary}

Throughout the numerical sections, if $F$ is the distribution function of
the law under discussion, $q_{0.05}=F^{-1}(0.95)$ denotes its upper $5\%$
critical value; similarly $q_{0.10}=F^{-1}(0.90)$.
Table~\ref{tab:strata-isotropic} illustrates the separation for $q=3$ and
$r=2$.  The lower-stratum values were obtained from $10^6$ standard Gaussian
matrices.  On the rank-zero stratum the statistic has more than $95\%$ mass
at zero, so its $95\%$ quantile is zero.

\begin{table}[t]
\caption{Isotropic stratified limits for $q=3$, $r=2$; Monte Carlo standard
errors for positive quantiles are below $1.5\times10^{-3}$.}
\label{tab:strata-isotropic}
\centering
\begin{tabular}{cccccc}
\toprule
$s$ & $k_s$ & $m_s$ & statistic & $q_{0.10}$ & $q_{0.05}$\\
\midrule
$2$ & $1$ & $0$ & $(\lambda_1^+)^2$ & $1.6424$ & $2.7055$\\
$1$ & $2$ & $1$ & $(\lambda_2^+)^2$ & $0.0356$ & $0.2304$\\
$0$ & $3$ & $2$ & $(\lambda_3^+)^2$ & $0$ & $0$\\
\bottomrule
\end{tabular}
\end{table}

The isotropic theorem is not an anisotropic ordering principle.  There is,
however, one anisotropic case in which the geometry alone is sufficient.

\begin{proposition}[Anisotropic transition dominance in active corank one]
\label{prop:corank-one-transition}
Let $k-m=1$, let $\mathcal S$ be any positive-definite operator on
$\Smat^k$, and put
\begin{equation}
\mathcal C=\mathcal S^{-1/2}(\Splus^k),
\qquad
\mathcal D=\mathcal S^{-1/2}(\mathcal K_{k-1}^k)=\partial\mathcal C.
\end{equation}
For every null drift $\mu\in\mathcal D$ and standard Gaussian $Y$,
\begin{equation}
\dist_F^2(Y+\mu,\mathcal D)-\dist_F^2(Y+\mu,\mathcal C)
\preceq_{\mathrm{st}}(Z_+)^2,
\qquad Z\sim N(0,1).
\label{eq:corank-one-transition-dominance}
\end{equation}
Indeed, the two variables can be coupled so that the left-hand side is at
most $(Z_+)^2$ almost surely.  Thus top-stratum calibration controls all
local null transitions whenever $q-r=1$, without an isotropy assumption.
\end{proposition}

\begin{proof}
The set $\mathcal D$ is the boundary of the closed convex cone $\mathcal C$.
Choose a unit supporting normal $n$ at $\mu$, meaning
$\langle n,z-\mu\rangle\le0$ for every $z\in\mathcal C$.  Outside $\mathcal C$ the two
projections coincide on the boundary, while inside $\mathcal C$ the distance
to the boundary is no larger than the positive displacement along this
normal.  The latter is the positive part of a standard normal coordinate.
The full argument is given in Section~\ref{subsec:supp-corank-one} of the supplementary material.
\end{proof}

\subsection{Computation on anisotropic lower strata}

The active reduction remains computationally useful even when $s<r$.  For a
given standard Gaussian draw $Y$, the alternative term in
\eqref{eq:stratified-limit} is the convex semidefinite projection
\begin{equation}
\min_{B\succeq0}
\|\mathcal S_{P_s}^{-1/2}B-Y\|_F^2,
\label{eq:lower-alt-projection}
\end{equation}
whereas the null term adds $\rank(B)\le m_s$.  A convenient factorisation is
$B=ZZ^{\top}$ with $Z\in\mathbb R^{k_s\times m_s}$, giving
\begin{equation}
\min_{Z\in\mathbb R^{k_s\times m_s}}
\|\mathcal S_{P_s}^{-1/2}(ZZ^{\top})-Y\|_F^2.
\label{eq:lower-rank-factor}
\end{equation}
Problem \eqref{eq:lower-rank-factor} is nonconvex and invariant under
$Z\mapsto ZO$, $O\in\operatorname O(m_s)$.  Its global minimum is nevertheless
the correct null distance, and the active dimension $k_s=q-s$ is often much
smaller than the ambient statistical dimension.  In low dimensions it can
be computed by multistart factor optimisation; in the isotropic case the
spectral truncation in Proposition~\ref{prop:isotropic-spectral} gives the
global solution directly.

The nonconvexity has two statistical consequences.  First, the limiting law
is generally not a chi-bar-square mixture, because the rank-constrained set
is not convex and has no conic intrinsic-volume decomposition of the form
used in Corollary~\ref{cor:chibar}.  Second, sensitivity at lower strata may
fail when the best rank-$m_s$ approximation is nonunique, for example at an
eigenvalue tie in the isotropic problem.  Extending
Section~\ref{sec:sensitivity} to lower strata therefore requires a
set-valued or directional theory rather than a direct reuse of the smooth
top-stratum formula.

For Monte Carlo calibration, the two optimisation problems should be solved
with the same Gaussian draw and the statistic computed as their difference.
The inequality
\begin{equation}
0\le
\Delta_{s,r}(P_s)
\le
\dist_F^2\{Y,\mathcal D_{P_s,m_s}\}
\le K(1+\|Y\|_F^2)
\label{eq:lower-bound-computation}
\end{equation}
for a finite constant $K$ that can be chosen uniformly when $\mathcal S_{P_s}$ ranges over a compact
positive-definite set.  It supplies a useful numerical check and the
domination used in Proposition~\ref{prop:law-continuity}.

\subsection{A stratified envelope and the remaining interface problem}

For a null parameter $\vartheta$ of rank $s$, let
$q_{\alpha,s}(\vartheta)$ be the upper-$\alpha$ critical value of its
fixed-stratum limit law.  Let $\mathcal K_s$ be a compact subset of the
rank-$s$ stratum that is rank separated, meaning that the smallest positive
eigenvalue of the matrix component is bounded away from zero uniformly over
$\mathcal K_s$.  Define
\begin{equation}
c_\alpha^{\mathrm{str}}
=
\max_{0\le s\le r}
\sup_{\vartheta\in\mathcal K_s}q_{\alpha,s}(\vartheta).
\label{eq:stratified-envelope}
\end{equation}

\begin{remark}[Uniformity on separated compact strata]
\label{rem:stratified-uniformity}
Suppose that the LAN expansion, local-set convergence, stochastic
localisation, and convergence of the likelihood-ratio distribution are
uniform on each $\mathcal K_s$.  If $c_\alpha^{\mathrm{str}}$ is a continuity
point of the corresponding limit laws, then
\begin{equation}
\limsup_{n\to\infty}
\sup_{\vartheta\in\cup_{s=0}^r\mathcal K_s}
\Pr_\vartheta\{\Lambda_n>c_\alpha^{\mathrm{str}}\}
\le\alpha.
\label{eq:stratified-size}
\end{equation}
\end{remark}

Rank separation excludes eigenvalues that vanish at the local statistical
rate.  Corollary~\ref{cor:null-transition} supplies the missing transition
experiment for such paths.  Theorem~\ref{thm:isotropic-dominance} then closes
the interface problem under isotropy, and
Proposition~\ref{prop:corank-one-transition} closes it under arbitrary
anisotropy when $q-r=1$.  In a general anisotropic active block of dimension
at least two, neither the fixed-stratum envelope nor the endpoint of a
particular transition is presently known to dominate every translated
nonconvex law.  Within the present framework, full uniformity on the
untrimmed null therefore remains unresolved in that genuinely anisotropic
regime.

\section{Positive-level sensitivity on the top stratum}
\label{sec:sensitivity}

The lower-stratum statistic involves a nonconvex rank-constrained projection.
We therefore develop sensitivity only on the top stratum $s=r$, where the
canonical limit is the convex projection in \eqref{eq:top-projection}.  A
useful simplification occurs here: the derivative of the optimal value does
not require differentiability of the projection point.  Uniqueness,
coercivity (the objective tends to infinity as $\|B\|_F\to\infty$), and
an envelope theorem are sufficient.

\subsection{Projection-value derivative}

For maps between finite-dimensional normed spaces, $\mathrm df(x)[h]$
denotes the Fr\'echet derivative: the unique linear map in $h$ such that
$f(x+h)=f(x)+\mathrm df(x)[h]+o(\|h\|)$.  Let $\mathfrak P_k$ be the cone of positive-definite self-adjoint operators
on $\Smat^k$.  For $\mathcal S\in\mathfrak P_k$, set
$\mathcal R=\mathcal S^{-1/2}$ and define
\begin{equation}
V(\mathcal R,Y)
=
\frac12\min_{B\succeq0}\|\mathcal RB-Y\|_F^2,
\qquad
T(\mathcal R,Y)
=
\|Y\|_F^2-2V(\mathcal R,Y).
\label{eq:value-problem}
\end{equation}
The minimiser $B^*$ is unique because $\mathcal R$ is invertible and the
quadratic objective is strongly convex, meaning that its Hessian is bounded
below by a positive multiple of the identity.  The following Karush--Kuhn--
Tucker conditions are the first-order optimality conditions: there is a dual
matrix $\Lambda^*\succeq0$ such that
\begin{equation}
\mathcal R^*(\mathcal RB^*-Y)-\Lambda^*=0,
\quad
B^*\succeq0,
\quad
\Lambda^*\succeq0,
\quad
\langle B^*,\Lambda^*\rangle=0.
\label{eq:KKT}
\end{equation}

\begin{proposition}[Operator differentiability of the projection value]
\label{prop:value-sensitivity}
For every invertible $\mathcal R$ and every $Y\in\Smat^k$, the minimiser
$B^*(\mathcal R,Y)$ in \eqref{eq:value-problem} is unique and continuous in
$(\mathcal R,Y)$.  The maps $V$ and $T$ are continuously Fr\'echet
differentiable with respect to $\mathcal R$, and
\begin{equation}
\mathrm dT(\mathcal R,Y)[\dot{\mathcal R}]
=
-2\left\langle
\mathcal RB^*-Y,
\dot{\mathcal R}B^*
\right\rangle.
\label{eq:dT}
\end{equation}
No strong-regularity assumption on the KKT solution mapping is required.
\end{proposition}

\begin{proof}
On a sufficiently small neighbourhood of an invertible $\mathcal R_0$, the
smallest singular value of $\mathcal R$ is bounded below by a positive
constant.  Hence the objective in \eqref{eq:value-problem} is uniformly
strongly convex and uniformly coercive in $B$.  The minimiser is therefore
unique and remains in a common compact set when $(\mathcal R,Y)$ ranges over
a compact neighbourhood.  Standard parametric-minimisation arguments imply
continuity of $B^*(\mathcal R,Y)$.  Since the feasible cone is fixed and the
objective is continuously differentiable, the envelope theorem gives
\begin{equation}
\mathrm dV(\mathcal R,Y)[\dot{\mathcal R}]
=
\left\langle
\mathcal RB^*-Y,
\dot{\mathcal R}B^*
\right\rangle.
\end{equation}
Continuity of the right-hand side yields continuous Fr\'echet
differentiability.  Equation \eqref{eq:dT} follows from the definition of
$T$.
\end{proof}

The value-differentiation step is an application of standard
parametric-optimisation and envelope arguments; see
\citet{BonnansShapiro2000}.  It is important to distinguish differentiability
of the optimal value from differentiability of the metric projection or of
the full primal--dual solution map \citep{Shapiro2016}.  Strong regularity of
generalized equations goes back to \citet{Robinson1980}; its
semidefinite-programming formulation is discussed by \citet{ChanSun2008}.
Related statistical perturbation questions for semidefinite programs are
treated by \citet{Shapiro2019}.  Differentiation and generalized derivatives
of probability functions under moving inequality-defined sets are developed,
from complementary viewpoints, by \citet{Uryasev1994} and
\citet{HantouteHenrionPerezAros2017}.
The critical-value calculation below only requires the value derivative in
Proposition~\ref{prop:value-sensitivity}; strong regularity is needed only if
one also wants derivatives of $(B^*,\Lambda^*)$.

For linear operators $\mathcal L_1,\mathcal L_2$ on $\Smat^k$, the
Hilbert--Schmidt inner product is
$\langle\mathcal L_1,\mathcal L_2\rangle_{\mathrm{HS}}
=\tr(\mathcal L_1^*\mathcal L_2)$, with associated norm
$\|\cdot\|_{\mathrm{HS}}$.  Let $E^*=\mathcal RB^*-Y$ and define
\begin{equation}
M_T(\mathcal R,Y)
=
-\{E^*\otimes B^*+B^*\otimes E^*\},
\label{eq:MT}
\end{equation}
where $(A\otimes B)(H)=\langle B,H\rangle A$.  For self-adjoint
perturbations,
\begin{equation}
\mathrm dT(\mathcal R,Y)[\dot{\mathcal R}]
=
\langle M_T(\mathcal R,Y),\dot{\mathcal R}\rangle_{\mathrm{HS}}.
\label{eq:dT-HS}
\end{equation}

\subsection{Distribution and quantile derivatives}

Let
\begin{equation}
F_{\mathcal R}(c)
=
\Pr\{T(\mathcal R,Y)\le c\},
\qquad
q_\alpha(\mathcal R)
=
F_{\mathcal R}^{-1}(1-\alpha),
\label{eq:Fq}
\end{equation}
where
$F^{-1}(u)=\inf\{c\in\mathbb R:F(c)\ge u\}$.
The law may have an atom at zero, so differentiation is restricted to
positive levels.

The positive-level geometry is simpler than it first appears.  If
$C=\mathcal R(\Splus^k)$, then
\begin{equation}
T(\mathcal R,Y)
=
\|\Pi_C(Y)\|_F^2
=
\dist_F^2(Y,C^{\circ}).
\label{eq:T-polar-distance}
\end{equation}
Squared distance to a closed convex set is continuously differentiable.
Therefore
\begin{equation}
\nabla_YT(\mathcal R,Y)
=
2\Pi_C(Y)
=
2\mathcal RB^*(\mathcal R,Y).
\label{eq:gradient-Y}
\end{equation}
On the level set $T=c>0$,
$\|\nabla_YT\|_F=2\sqrt c$.  Thus every positive level is regular;
no separate transversality assumption is needed.

Let $d=k(k+1)/2$, let
$\phi(Y)=(2\pi)^{-d/2}\exp(-\|Y\|_F^2/2)$ be the standard Gaussian density
in Frobenius coordinates, and let $\mathcal H^{d-1}$ denote
$(d-1)$-dimensional Hausdorff surface measure.  For a self-adjoint direction
$D$, define the level-set functional
\begin{equation}
h_{\widetilde{\mathcal R},D}(t)
=
\int_{\{Y:T(\widetilde{\mathcal R},Y)=t\}}
\frac{
\mathrm dT(\widetilde{\mathcal R},Y)[D]
}{
\|\nabla_YT(\widetilde{\mathcal R},Y)\|_F
}
\phi(Y)\,\mathrm d\mathcal H^{d-1}(Y),
\label{eq:level-functional}
\end{equation}

\begin{assumption}[Conditional level-set regularity]
\label{ass:surface-domination}
For the positive level $c$ and operator $\mathcal R$ under consideration,
there are neighbourhoods $\mathcal U$ of $\mathcal R$ and
$J\subset(0,\infty)$ of $c$ such that:
\begin{enumerate}
\item $h_{\widetilde{\mathcal R},D}(t)$ is finite for
$\widetilde{\mathcal R}\in\mathcal U$, $t\in J$, and every
self-adjoint $D$;

\item the map in \eqref{eq:level-functional} is linear and continuous in
$D$, and
\begin{equation}
\sup_{\|D\|_{\mathrm{HS}}\le1}
\left|
h_{\widetilde{\mathcal R},D}(t)-h_{\mathcal R,D}(c)
\right|
\longrightarrow0
\label{eq:uniform-level-continuity}
\end{equation}
whenever $(\widetilde{\mathcal R},t)\to(\mathcal R,c)$;

\item the surface integrals in \eqref{eq:level-functional} are uniformly
tight in the Gaussian tails over
$\widetilde{\mathcal R}\in\mathcal U$, $t\in J$, and
$\|D\|_{\mathrm{HS}}\le1$.
\end{enumerate}
\end{assumption}

\begin{theorem}[Conditional positive-level Hadamard shape derivative]
\label{thm:quantile-sensitivity}
Under Assumption~\ref{ass:surface-domination}, the map
$\mathcal R\mapsto F_{\mathcal R}(c)$ is Hadamard differentiable,
tangentially to the self-adjoint operator space, and
\begin{equation}
\mathrm d_HF_{\mathcal R}(c)[D]
=
-h_{\mathcal R,D}(c)
=
-
\frac{1}{2\sqrt c}
\int_{\{Y:T(\mathcal R,Y)=c\}}
\mathrm dT(\mathcal R,Y)[D]
\phi(Y)\,\mathrm d\mathcal H^{d-1}(Y).
\label{eq:dF}
\end{equation}
That is, for every $t_n\to0$ and every sequence of self-adjoint
$D_n\to D$ such that $\mathcal R+t_nD_n$ remains invertible,
\begin{equation}
\frac{
F_{\mathcal R+t_nD_n}(c)-F_{\mathcal R}(c)
}{t_n}
\longrightarrow
\mathrm d_HF_{\mathcal R}(c)[D].
\label{eq:Hadamard-definition}
\end{equation}
If $q_\alpha(\mathcal R)>0$, the continuous component has a density
$f_{\mathcal R}$ that is continuous and positive at
$q_\alpha(\mathcal R)$, and Assumption~\ref{ass:surface-domination}
holds on a neighbourhood of that level, then
$\mathcal R\mapsto q_\alpha(\mathcal R)$ is Hadamard differentiable
with
\begin{equation}
\mathrm d_Hq_\alpha(\mathcal R)[D]
=
-
\frac{
\mathrm d_HF_{\mathcal R}(q_\alpha)[D]
}{
f_{\mathcal R}(q_\alpha)
}.
\label{eq:dq}
\end{equation}
\end{theorem}

\begin{proof}[Proof sketch]
Use the regular positive-level tube generated by
\eqref{eq:gradient-Y}.  The continuous Fr\'echet expansion of
$T(\mathcal R,Y)$ in the operator argument moves the boundary by
$-t\,\mathrm dT(\mathcal R,Y)[D]/\|\nabla_YT\|_F$ in its normal
direction.  The uniform continuity in
\eqref{eq:uniform-level-continuity} and the Gaussian-tail condition make
this expansion uniform for perturbation sequences $D_n\to D$.  Integration
in tubular coordinates gives \eqref{eq:dF} and
\eqref{eq:Hadamard-definition}.  The quantile formula follows from the
Hadamard implicit-function theorem.  This is only a proof sketch.  Lemma~\ref{lem:supp-regular-sublevel}
constructs the regular moving sublevel-set argument, and
Proposition~\ref{prop:supp-Hadamard-probability} applies it to the Gaussian
projection statistic and gives the complete proof.
\end{proof}

The theorem is deliberately conditional: Assumption~\ref{ass:surface-domination}
contains the uniform regularity needed to differentiate the probability of a
moving sublevel set.  We do not assert it uniformly over all
$\mathcal R\in\mathfrak P_k$.  Section~\ref{sec:supp-sensitivity} of the supplementary material proves a general regular-sublevel lemma and verifies each remaining ingredient
for the present statistic; continuity of the level functional is retained as
the explicit model-dependent condition.  For the family
$\mathcal C_\gamma$ in Section~\ref{sec:applications}, the distribution is an
explicit differentiable chi-bar-square mixture, providing an independent
verification of the derivative used numerically.

\subsection{Chain rule and geometric least-favourable orientation}

Return to a rank-$r$ null point and put $k=q-r$.  We represent the
Grassmann manifold $\Gr(k,q)$ as the set of rank-$k$ orthogonal projectors
$P=P^\top=P^2$ on $\mathbb R^q$.  Its tangent space at $P$ is
\begin{equation}
T_P\Gr(k,q)=
\{\dot P\in\Smat^q:P\dot PP=0,\ (I-P)\dot P(I-P)=0\},
\label{eq:Grassmann-tangent-definition}
\end{equation}
so tangent matrices have only off-diagonal blocks relative to
$\operatorname{range}(P)\oplus\ker(P)$.  Choose a smooth local orthonormal
frame $U(P)\in\mathbb R^{q\times k}$ satisfying $U(P)^\top U(P)=I_k$ and
$U(P)U(P)^\top=P$.  For $\dot P\in T_P\Gr(k,q)$, the horizontal lift used
here is $\dot U=\dot P U$.  Define
\begin{equation}
\mathcal A_P(H)=U(P)^\top H U(P),
\quad
\mathcal S_P=\mathcal A_P\mathcal I_{\mathrm{eff}}^{-1}\mathcal A_P^*,
\quad
\mathcal R_P=\mathcal S_P^{-1/2}.
\label{eq:Grassmann-active-definitions}
\end{equation}
At fixed $\mathcal I_{\mathrm{eff}}$,
\begin{equation}
\dot{\mathcal S}_P
=
\dot{\mathcal A}_P\mathcal I_{\mathrm{eff}}^{-1}\mathcal A_P^*
+
\mathcal A_P\mathcal I_{\mathrm{eff}}^{-1}\dot{\mathcal A}_P^*,
\label{eq:dSP}
\end{equation}
where
$\dot{\mathcal A}_P(H)=\dot U^{\top}HU+U^{\top}H\dot U$.
If $\mathcal L_P=\mathcal S_P^{1/2}$, let
$\dot{\mathcal L}_P$ denote the directional derivative of the operator
square-root map at $\mathcal S_P$ in direction $\dot{\mathcal S}_P$.  It is
the unique solution of
\begin{equation}
\mathcal L_P\dot{\mathcal L}_P
+
\dot{\mathcal L}_P\mathcal L_P
=
\dot{\mathcal S}_P,
\qquad
\dot{\mathcal R}_P
=
-\mathcal R_P\dot{\mathcal L}_P\mathcal R_P.
\label{eq:dR}
\end{equation}
Combining \eqref{eq:dq}--\eqref{eq:dR} identifies the Hadamard directional
derivative $\mathrm d_Hq_\alpha(P)[\dot P]$.  If the level-set functional in
\eqref{eq:dF} can be evaluated and its pullback admits an ambient Riesz
representative $\Gamma_P\in\Smat^q$, meaning that the ambient differential
has the form $\langle\Gamma_P,\dot P\rangle_F$, then, with the commutator
$[A,B]=AB-BA$, the intrinsic gradient for the Frobenius metric on projectors
is
\begin{equation}
\operatorname{grad}q_\alpha(P)
=
[P,[P,\Gamma_P]].
\label{eq:grassmann-gradient}
\end{equation}
Thus the commutator formula is the geometric projection of an available
ambient differential.  A general Monte Carlo estimator of the surface
functional, and hence a fully operational high-dimensional Grassmannian
algorithm, is beyond the scope of this paper; the numerical example below
uses an independent closed-form derivative.

For a compact admissible set $\mathcal G\subset\Gr(k,q)$, define the
geometric least-favourable value
\begin{equation}
q_\alpha^{\mathrm{geo}}
=
\sup_{P\in\mathcal G}q_\alpha(P).
\label{eq:geo-LF}
\end{equation}
Continuity gives existence of a maximiser.  At an interior differentiable
maximiser $P_*$, \eqref{eq:grassmann-gradient} vanishes, equivalently
$(I-P_*)\Gamma_{P_*}P_*=0$.  This problem varies orientation at fixed
information; it is not the same as the stratified envelope
\eqref{eq:stratified-envelope}, where the true rank and the information may
both vary.

\section{Applications and numerical validation}
\label{sec:applications}

\subsection{General Gaussian covariance-component reduction}

The model-specific front end can be made explicit for a broad class of
Gaussian covariance models.  Let $X_1,\ldots,X_n$ be independent and
identically distributed with
\begin{equation}
X_i\sim\mathcal N_m\{0,V(\psi,\Sigma)\},
\qquad
V(\psi,\Sigma)=R(\psi)+L\Sigma L^{\top},
\label{eq:general-gaussian-model}
\end{equation}
where $R(\psi_0)\succ0$ and $L\in\mathbb R^{m\times q}$ is fixed.  Define
the local covariance derivative maps
\begin{equation}
\mathcal D_\psi a
=\mathrm dR(\psi_0)[a],
\qquad
\mathcal D_\Sigma H=LHL^{\top},
\label{eq:derivative-maps}
\end{equation}
where $\mathrm dR(\psi_0)[a]$ is the directional derivative of the smooth
map $R$ at $\psi_0$ in direction $a$, and equip $\Smat^m$ with the Gaussian Fisher inner product
\begin{equation}
\langle A,B\rangle_{V_0}
=
\frac12\tr(V_0^{-1}AV_0^{-1}B),
\qquad
V_0=V(\psi_0,\Sigma_0).
\label{eq:Fisher-inner}
\end{equation}
Let $\mathcal D_\psi^*$ and $\mathcal D_\Sigma^*$ denote the adjoints with
respect to this inner product and the Euclidean--Frobenius inner products.
If $\mathcal D_\psi^*\mathcal D_\psi$ is invertible, define the
$V_0$-orthogonal projector onto the nuisance covariance tangent space by
\begin{equation}
\Pi_\psi
=
\mathcal D_\psi
(\mathcal D_\psi^*\mathcal D_\psi)^{-1}
\mathcal D_\psi^*.
\label{eq:nuisance-projector}
\end{equation}

\begin{proposition}[Efficient information in Gaussian covariance models]
\label{prop:Gaussian-Ieff}
For \eqref{eq:general-gaussian-model},
\begin{equation}
\mathcal I_{\mathrm{eff}}
=
\mathcal D_\Sigma^*(I-\Pi_\psi)\mathcal D_\Sigma.
\label{eq:Gaussian-Ieff}
\end{equation}
Equivalently,
\begin{equation}
\langle H,\mathcal I_{\mathrm{eff}}K\rangle
=
\left\langle
(I-\Pi_\psi)\mathcal D_\Sigma H,
(I-\Pi_\psi)\mathcal D_\Sigma K
\right\rangle_{V_0}.
\label{eq:Gaussian-Ieff-form}
\end{equation}
The operator is positive definite precisely when
$(I-\Pi_\psi)\mathcal D_\Sigma$ is injective.
\end{proposition}

\begin{proof}
The full information blocks are
$\mathcal I_{\psi\psi}=\mathcal D_\psi^*\mathcal D_\psi$,
$\mathcal I_{\psi\Sigma}=\mathcal D_\psi^*\mathcal D_\Sigma$, and
$\mathcal I_{\Sigma\Sigma}=\mathcal D_\Sigma^*\mathcal D_\Sigma$.
Substitution in the Schur complement \eqref{eq:efficient-information}
gives \eqref{eq:Gaussian-Ieff}.  Since $I-\Pi_\psi$ is an orthogonal
projector, \eqref{eq:Gaussian-Ieff-form} and the injectivity statement
follow.
\end{proof}

Subsection~\ref{subsec:supp-Gaussian-information} records the trace
calculation and the adjoint formulas underlying this proposition.

This formula includes linear mixed-effects covariance components after
profiling residual covariance parameters.  Gaussian mean parameters are
Fisher-orthogonal to covariance directions and therefore do not alter
\eqref{eq:Gaussian-Ieff}.  The example below is deliberately simple enough
to permit an exact finite-sample likelihood calculation while retaining a
nontrivial covariance nuisance projection.

\subsection{A Gaussian covariance component with residual-variance nuisance}

Let $e_j$ denote the $j$th standard coordinate vector, set
$J=(e_1,e_2)\in\mathbb R^{3\times2}$, and consider
\begin{equation}
X_i\sim\mathcal N_3\{0,V(\tau,\Sigma)\},
\qquad
V(\tau,\Sigma)=\tau I_3+J\Sigma J^{\top},
\qquad
\tau>0,
\quad
\Sigma\succeq0.
\label{eq:nuisance-model}
\end{equation}
We test $H_0:\Sigma=0$ against $H_1:\Sigma\succeq0$ at
$(\tau_0,\Sigma_0)=(1,0)$.  The third coordinate identifies the residual
variance, so the nuisance and semidefinite components are not confounded.

In the orthonormal basis
$E_{11},E_{22},F_{12}=(e_1e_2^{\top}+e_2e_1^{\top})/\sqrt2$ of
$\Smat^2$, the unprofiled matrix information is $\frac12\Id$, the nuisance
information is $3/2$, and the cross-information vector is
$(1/2,1/2,0)^{\top}$.  Hence
\begin{equation}
\mathcal I_{\mathrm{eff}}
=
\begin{pmatrix}
1/3 & -1/6 & 0\\
-1/6 & 1/3 & 0\\
0 & 0 & 1/2
\end{pmatrix}.
\label{eq:nuisance-Ieff}
\end{equation}
For $B=x_1E_{11}+x_2E_{22}+x_3F_{12}\in\Smat^2$, define the
orthonormal Lorentz coordinates
$u=(x_1+x_2)/\sqrt2$, $v=(x_1-x_2)/\sqrt2$, and $w=x_3$.  In these
coordinates,
\begin{equation}
\mathcal I_{\mathrm{eff}}
=
\diag(1/6,1/2,1/2),
\qquad
\mathcal S
=
\diag(6,2,2).
\label{eq:nuisance-S}
\end{equation}
Thus the whitened cone is the circular cone
$u\ge (v^2+w^2)^{1/2}/\sqrt3$, with half-angle $\pi/3$.  Its intrinsic
volumes are
\begin{equation}
(v_0,v_1,v_2,v_3)
=
\left(
\frac{2-\sqrt3}{4},
\frac14,
\frac{\sqrt3}{4},
\frac14
\right),
\label{eq:nuisance-weights}
\end{equation}
where $v_2=1/2-v_0=\sqrt3/4$.  Thus
$(v_0,v_1,v_2,v_3)=(0.06699,0.25,0.43301,0.25)$, and the asymptotic
$5\%$ critical value is $6.1252334478$.

The finite-sample likelihood can be profiled exactly.  Let
$\widehat V_n=n^{-1}\sum_{i=1}^nX_iX_i^\top$ be the sample covariance,
let $d_1\ge d_2$ be the eigenvalues of its upper $2\times2$ block, and let
$d_3=(\widehat V_n)_{33}$ be the third sample variance.  Under the null,
$\widehat\tau_0=(d_1+d_2+d_3)/3$.  Under the alternative, for fixed
$t=\tau$ the optimal upper-block eigenvalues are
$\max(d_j,t)$, and the profiled negative log-likelihood is
\begin{equation}
g(t)
=
\log t+\frac{d_3}{t}
+
\sum_{j=1}^2
\left
\{
\log\max(d_j,t)
+
\frac{d_j}{\max(d_j,t)}
\right\}.
\label{eq:profile-g}
\end{equation}
Consequently,
\begin{equation}
\Lambda_n
=
n\left
\{
3\log\widehat\tau_0+3-
\min_{t>0}g(t)
\right\}.
\label{eq:exact-LRT-example}
\end{equation}

Table~\ref{tab:finite-size} reports rejection probabilities under the null
using the asymptotic critical value.  Each row uses $10^6$ independent
replications; Monte Carlo standard errors are about $2.2\times10^{-4}$.
The test is mildly conservative at small samples and approaches the nominal
level.

\begin{table}[t]
\caption{Finite-sample size in model \eqref{eq:nuisance-model} at nominal
level $5\%$ using critical value $6.1252334478$.}
\label{tab:finite-size}
\centering
\begin{tabular}{rrrr}
\toprule
$n$ & rejection probability & Monte Carlo s.e. & empirical $0.95$ quantile\\
\midrule
20   & 0.048668 & 0.000215 & 6.067122\\
50   & 0.048358 & 0.000215 & 6.054099\\
100  & 0.048242 & 0.000214 & 6.047978\\
250  & 0.048757 & 0.000215 & 6.070754\\
1000 & 0.049050 & 0.000216 & 6.085415\\
\bottomrule
\end{tabular}
\end{table}

This example demonstrates that nuisance profiling is not merely cosmetic:
without the Schur complement, the active covariance would be proportional
to the identity and the weights would be those of the self-dual
$\Splus^2$ cone rather than \eqref{eq:nuisance-weights}.

\subsection{A lower-stratum likelihood ratio and rank transitions}
\label{subsec:lower-stratum-finite}

To validate the nonconvex fixed-stratum limit, let
$X_1,\ldots,X_n$ be independent and identically distributed according to
\begin{equation}
X_i\sim\mathcal N_3(0,I_3+\Sigma),
\qquad
H_0:\ \Sigma\succeq0,\ \rank(\Sigma)\le1,
\qquad
H_1:\ \Sigma\succeq0,
\label{eq:lower-stratum-model}
\end{equation}
with known residual covariance $I_3$.  At $\Sigma_0=0$ the true rank is
$s=0<r=1$.  If $d_1\ge d_2\ge d_3$ are the eigenvalues of the sample
covariance $n^{-1}\sum_iX_iX_i^\top$, direct spectral profiling gives the
finite-sample likelihood ratio
\begin{equation}
\Lambda_n
=
n\sum_{j=2}^3
\{d_j-1-\log d_j\}\mathbf 1_{\{d_j>1\}}.
\label{eq:lower-stratum-LRT}
\end{equation}
The central fixed-stratum limit is
\begin{equation}
\Delta_{0,1}^{(0)}
=
\sum_{j=2}^3\{\lambda_j(Y)_+\}^2,
\label{eq:lower-central-limit}
\end{equation}
where $Y$ is standard Gaussian in $\Smat^3$ and its eigenvalues are
ordered decreasingly.  Its simulated $0.95$ quantile is $1.369475$.  For the
rank-one top stratum, the seed-9002 Monte Carlo diagnostic is $5.490850$.
The same active-dimension-two law is the $\gamma=1$ member of the analytic
chi-bar-square family below, whose $0.95$ quantile is
$5.4845131865391$.  The analytic value is used for rejection decisions; the
Monte Carlo value is retained only as a reproducibility diagnostic.

Now take the null triangular array
\begin{equation}
\Sigma_n=\diag(c/\sqrt n,0,0),
\qquad c\ge0.
\label{eq:transition-array}
\end{equation}
The local-alternative result gives the transition law
\begin{equation}
\Delta_{0,1}^{(c)}
=
\sum_{j=2}^3
\left[
\lambda_j\left(
Y+\frac{c}{\sqrt2}E_{11}
\right)_+
\right]^2.
\label{eq:transition-limit}
\end{equation}
Its $0.95$ quantiles are $1.369475$, $2.731123$, and $3.988097$ for
$c=0,2,4$, respectively.  Thus the central lower-stratum critical value is
not uniform along the interface.  Theorem~\ref{thm:isotropic-dominance}
shows, for every $c\ge0$, that the top-stratum value is conservative; the
simulation below illustrates this pathwise result in the exact likelihood.

Table~\ref{tab:lower-transition} compares the exact finite-sample
likelihood ratio with the corresponding transition law.  Transition
quantiles use $10^6$ Gaussian draws; each finite-sample row uses
$5\times10^5$ Wishart draws.  The column labelled $s=0$ uses the central
critical value $1.369475$, the column labelled top uses the analytic value
$5.4845131865391$, and the final column uses the $c$-specific transition
quantile.  Rejection probabilities are rounded to four decimal places.

\begin{table}[t]
\caption{Lower-stratum and transition calibration in
\eqref{eq:lower-stratum-model}.}
\label{tab:lower-transition}
\centering
\small
\begin{tabular}{rr@{\quad}rr@{\quad}rrr}
\toprule
$n$ & $c$ & empirical $q_{0.05}$ & limit $q_{0.05}$
& rejection at $s=0$ & rejection at top & rejection at transition\\
\midrule
50   & 0 & 1.077873 & 1.369475 & 0.0336 & 0.0003 & 0.0336\\
50   & 2 & 2.210077 & 2.731123 & 0.1089 & 0.0028 & 0.0313\\
50   & 4 & 3.091132 & 3.988097 & 0.1788 & 0.0092 & 0.0263\\
200  & 0 & 1.219466 & 1.369475 & 0.0413 & 0.0004 & 0.0413\\
200  & 2 & 2.454491 & 2.731123 & 0.1319 & 0.0040 & 0.0393\\
200  & 4 & 3.512720 & 3.988097 & 0.2186 & 0.0130 & 0.0360\\
1000 & 0 & 1.289099 & 1.369475 & 0.0452 & 0.0005 & 0.0452\\
1000 & 2 & 2.613732 & 2.731123 & 0.1461 & 0.0046 & 0.0453\\
1000 & 4 & 3.766082 & 3.988097 & 0.2431 & 0.0161 & 0.0433\\
\bottomrule
\end{tabular}
\end{table}

The experiment confirms both aspects of the stratified theory: the
nonconvex lower-stratum law is the correct pointwise approximation, and a
local eigenvalue of order $n^{-1/2}$ changes the limit continuously away
from that central law.  The accompanying R script, the exact random seeds,
and the correspondence between generated CSV files and reported tables are
documented in Section~\ref{sec:supp-reproducibility} of the supplementary material.

\subsection{An anisotropic active-dimension-two family}

Consider the Gaussian matrix-location observation
$Y\sim\mathcal N_{\Smat^3}(0,\mathcal V)$, where the covariance operator
$\mathcal V$ is diagonal in the Frobenius-orthonormal basis $E_{ii}$ and
$F_{ij}$, with unit variances except
\begin{equation}
\mathcal VF_{12}=\gamma_{12}F_{12},
\qquad
\mathcal VF_{13}=\gamma_{13}F_{13}.
\label{eq:anisotropic-V}
\end{equation}
For the kernel plane
$K_\theta=\operatorname{span}\{e_1,\cos\theta\,e_2+\sin\theta\,e_3\}$, the reduced operator acts on $\Smat^2$.  Identifying
$B=x_1E_{11}+x_2E_{22}+x_3F_{12}$ with the Lorentz coordinates
$u=(x_1+x_2)/\sqrt2$, $v=(x_1-x_2)/\sqrt2$, and $w=x_3$, that operator is
\begin{equation}
\mathcal S_\theta
=
\diag\{1,1,\gamma(\theta)\},
\qquad
\gamma(\theta)
=
\gamma_{12}\cos^2\theta+
\gamma_{13}\sin^2\theta.
\label{eq:gamma-theta}
\end{equation}
The canonical cone is
\begin{equation}
\mathcal C_\gamma
=
\{(u,v,w):u\ge(v^2+\gamma w^2)^{1/2}\}.
\label{eq:Cgamma}
\end{equation}
Its solid angle $\Omega(\gamma)$, meaning the surface area of
$\mathcal C_\gamma\cap\{x\in\mathbb R^3:\|x\|_2=1\}$ on the unit sphere,
is
\begin{equation}
\Omega(\gamma)
=
\int_0^{2\pi}
\left
\{
1-
\frac{a_\gamma(\phi)}
{\sqrt{1+a_\gamma(\phi)^2}}
\right\}\,\mathrm d\phi,
\qquad
 a_\gamma(\phi)
=
\sqrt{\cos^2\phi+\gamma\sin^2\phi}.
\label{eq:Omega}
\end{equation}
The weights are
\begin{equation}
\begin{aligned}
v_3(\gamma)&=\Omega(\gamma)/(4\pi),
&v_0(\gamma)&=\Omega(\gamma^{-1})/(4\pi),\\
v_1(\gamma)&=1/2-v_3(\gamma),
&v_2(\gamma)&=1/2-v_0(\gamma).
\end{aligned}
\label{eq:gamma-weights}
\end{equation}
Moreover,
\begin{equation}
\Omega'(\gamma)
=
-\frac12
\int_0^{2\pi}
\frac{\sin^2\phi}
{a_\gamma(\phi)\{1+a_\gamma(\phi)^2\}^{3/2}}
\,\mathrm d\phi<0.
\label{eq:Omega-prime}
\end{equation}
Thus orientation changes the limiting law whenever
$\gamma_{12}\ne\gamma_{13}$.  Moreover, if
$\gamma_2>\gamma_1$, then
$\mathcal C_{\gamma_2}\subset\mathcal C_{\gamma_1}$.  For the same
Gaussian draw, Moreau's identity \citep{Moreau1962} therefore gives
$\|\Pi_{\mathcal C_{\gamma_2}}Y\|_F^2\le
\|\Pi_{\mathcal C_{\gamma_1}}Y\|_F^2$.  Hence every upper quantile is nonincreasing in $\gamma$.  The weak
ordering is all that is required below.  Strict decrease holds at the
reported confidence level over the evaluated range, as confirmed by the
explicit derivative and numerical values, but no general statement about
strict ordering of every positive quantile is used.

\begin{table}[t]
\caption{Selected anisotropic chi-bar-square calibrations.}
\label{tab:gamma}
\centering
\begin{tabular}{c@{\quad}cccc@{\quad}c}
\toprule
$\gamma$ & $v_0$ & $v_1$ & $v_2$ & $v_3$ & $q_{0.05}$\\
\midrule
0.25 & 0.087789 & 0.299751 & 0.412211 & 0.200249 & 5.877087\\
1.00 & 0.146447 & 0.353553 & 0.353553 & 0.146447 & 5.484513\\
4.00 & 0.200249 & 0.412211 & 0.299751 & 0.087789 & 5.005473\\
\bottomrule
\end{tabular}
\end{table}

The same elliptic cone also gives a genuine anisotropic lower-stratum
experiment.  Take $q=2$, $r=1$, and $s=0$.  In the Lorentz coordinates of
\eqref{eq:Cgamma}, the alternative tangent is $\mathcal C_\gamma$ and the
null tangent is its rank-one boundary
\begin{equation}
\mathcal D_\gamma=\partial\mathcal C_\gamma.
\label{eq:lower-anisotropic-boundary}
\end{equation}
Consequently,
\begin{equation}
\Delta_\gamma^{\mathrm{low}}(Y)
=
\dist^2(Y,\mathcal D_\gamma)-\dist^2(Y,\mathcal C_\gamma).
\label{eq:lower-anisotropic-stat}
\end{equation}
For $Y\notin\mathcal C_\gamma$, the projection onto the cone lies on its
boundary and the two distances coincide.  For
$Y\in\operatorname{int}(\mathcal C_\gamma)$, the second distance vanishes
and the statistic is the squared distance to the boundary.  Hence the atom
at zero has mass $1-v_3(\gamma)$.  Table~\ref{tab:lower-anisotropic} reports
$10^6$ Gaussian replications for each value of $\gamma$.  The top-stratum
critical value in active dimension one is $2.705543$.

\begin{table}[t]
\caption{Anisotropic lower-stratum law for $q=2$, $r=1$, $s=0$.}
\label{tab:lower-anisotropic}
\centering
\small
\begin{tabular}{c@{\quad}cc@{\quad}cc@{\quad}c}
\toprule
$\gamma$ & empirical atom & theoretical atom & mean & $q_{0.05}$
& rejection at top\\
\midrule
0.25 & 0.800194 & 0.799751 & 0.071468 & 0.444867 & 0.001945\\
1.00 & 0.853997 & 0.853553 & 0.042965 & 0.227754 & 0.000832\\
4.00 & 0.911943 & 0.912211 & 0.012152 & 0.035953 & 0.000007\\
\bottomrule
\end{tabular}
\end{table}

This example directly exercises the nonconvex lower-stratum projection under
anisotropy.  Its law is visibly different from the top-stratum
chi-bar-square law.  Because $q-r=1$, Proposition~\ref{prop:corank-one-transition}
shows that the top critical value is conservative not only at the central
lower stratum reported in the table, but along every local null transition
for each of these operators.  No analogous ordering is asserted for a
general anisotropic active block of dimension at least two.

\subsection{Derivative and ascent checks}

For the family \eqref{eq:Cgamma}, the quantile derivative can be evaluated
without estimating a surface integral.  Let $F_{\chi_j^2}$ and
$f_{\chi_j^2}$ denote the distribution function and density of
$\chi_j^2$ for $j\ge1$, with $F_{\chi_0^2}(c)=\mathbf1_{\{c\ge0\}}$.
If $F_\gamma(c)=\sum_{j=0}^3v_j(\gamma)F_{\chi_j^2}(c)$, then
\begin{equation}
q_\alpha'(\gamma)
=
-
\frac{
\sum_{j=0}^3v_j'(\gamma)F_{\chi_j^2}\{q_\alpha(\gamma)\}
}{
\sum_{j=1}^3v_j(\gamma)f_{\chi_j^2}\{q_\alpha(\gamma)\}
}.
\label{eq:qprime-gamma}
\end{equation}
At $\gamma=1$ and $\alpha=0.05$, this gives
$q_{0.05}'(1)=-0.339124730022$.  Table~\ref{tab:FD} compares this value with
central finite differences.  The benchmark is obtained by analytic
differentiation of the mixture weights and CDF, not by a smaller-step
difference.  At $h=10^{-5}$ the centered finite difference differs by
$5.34\times10^{-11}$.  The observed quadratic error decay validates the
operator-to-quantile differentiation pipeline against a family-specific
deterministic formula.

\begin{table}[t]
\caption{Finite-difference validation of $q_{0.05}'(1)$.}
\label{tab:FD}
\centering
\begin{tabular}{rrrr}
\toprule
$h$ & central difference & analytic derivative & absolute error\\
\midrule
$10^{-1}$ & $-0.339853065$ & $-0.339124730$ & $7.283\times10^{-4}$\\
$5\times10^{-2}$ & $-0.339306318$ & $-0.339124730$ & $1.816\times10^{-4}$\\
$10^{-2}$ & $-0.339131987$ & $-0.339124730$ & $7.257\times10^{-6}$\\
$2\times10^{-3}$ & $-0.339125020$ & $-0.339124730$ & $2.903\times10^{-7}$\\
$10^{-4}$ & $-0.339124731$ & $-0.339124730$ & $6.973\times10^{-10}$\\
$10^{-5}$ & $-0.339124730075$ & $-0.339124730022$ & $5.340\times10^{-11}$\\
\bottomrule
\end{tabular}
\end{table}

Finally set $\gamma_{12}=0.25$ and $\gamma_{13}=4$.  Since
$q_{0.05}(\gamma)$ is nonincreasing, $K_0$ is a least-favourable endpoint
of the family \eqref{eq:gamma-theta}.  We run gradient ascent in $\theta$
with an Armijo line search: the initial trial step is one, rejected steps are
halved, the Armijo constant is $10^{-4}$, and the gradient tolerance is
$10^{-7}$.  The derivative is
\begin{equation}
\frac{\mathrm d}{\mathrm d\theta}q_{0.05}\{\gamma(\theta)\}
=
q_{0.05}'\{\gamma(\theta)\}
(\gamma_{13}-\gamma_{12})\sin(2\theta).
\label{eq:theta-gradient}
\end{equation}
All starts reach the same endpoint to within $2.6\times10^{-8}$ in the
canonical angle.  Two runs satisfy the gradient tolerance directly; in the
other two the Armijo search reaches the floating-point step floor with
$|\nabla q|<1.7\times10^{-7}$.  Table~\ref{tab:ascent} reports the executed
runs without relabelling these latter stops as formal convergence.

\begin{table}[t]
\caption{Armijo ascent for the one-dimensional kernel family.}
\label{tab:ascent}
\centering
\small
\begin{tabular}{rrrrrrl}
\toprule
$\theta_0$ & initial $q_{0.05}$ & iterations & backtracks & final $|\theta|$
& final $q_{0.05}$ & termination\\
\midrule
0.2 & 5.765555 & 32 & 64  & $1.52\times10^{-8}$ & 5.877087 & gradient tolerance\\
0.6 & 5.356733 & 47 & 408 & $2.51\times10^{-8}$ & 5.877087 & step floor\\
1.0 & 5.111931 & 47 & 420 & $2.51\times10^{-8}$ & 5.877087 & step floor\\
1.3 & 5.028070 & 31 & 55  & $1.38\times10^{-8}$ & 5.877087 & gradient tolerance\\
\bottomrule
\end{tabular}
\end{table}

The numerical experiments address four distinct claims: the nuisance model checks finite-sample calibration of a genuine likelihood
ratio; the fixed and transition experiments validate the stratified limits;
the anisotropic lower experiment exercises the nonconvex projection; and the
analytic/finite-difference comparison plus Armijo trace check the explicit
top-stratum sensitivity calculation.
The executable R code and generated numerical files accompany the arXiv submission; full-run replication counts, random seeds, and the correspondence with the tables above are documented in Section~\ref{sec:supp-reproducibility} of the supplementary material below.

\section{Discussion}
\label{sec:discussion}

The active-block reduction separates the universal rank geometry from the
model-specific information calculation.  At a top-stratum point it recovers
a chi-bar-square law.  At a lower-rank point the null tangent is a
rank-constrained nonconvex cone, but the same reduction still gives the
correct Gaussian distance functional.  Corollary~\ref{cor:null-transition}
shows that a rank change at the $n^{-1/2}$ scale does not require another
asymptotic construction: it translates this canonical experiment by the
active first-order increment.

This transition formulation sharpens the calibration question.  A critical
value computed from a central lower-stratum law is not stable along the
interface, as the exact Gaussian experiment demonstrates.  Under isotropy,
however, compression to the kernel of the local drift proves that the
rank-$r$ top-stratum law dominates every admissible transition, while
interlacing orders the fixed strata.  When $q-r=1$, the null active set is
the entire boundary of the alternative cone, and a supporting-hyperplane
argument gives the same conclusion for arbitrary anisotropy.  Within the present
framework, the remaining uniformity gap is therefore confined to anisotropic
active blocks of dimension at least two.

The reduced covariance $\mathcal S_{P_s}$ is a normal-form descriptor rather
than a minimal invariant.  Its value is modularity: nuisance profiling and
model geometry determine $\mathcal S_{P_s}$, after which fixed-stratum,
transition, and sensitivity calculations are performed in the active space.
On the top stratum, the projection-value derivative avoids strong regularity
of the full primal--dual solution map.  Differentiating probabilities still
requires the stated uniform level-set and Gaussian-tail conditions; the
explicit active-dimension-two family provides an independent check of the
resulting quantile derivative.

Two extensions remain open.  First, generic anisotropic transition ordering
would require control of translated nonconvex projections beyond the
corank-one boundary case.  Second, lower-stratum sensitivity must accommodate
nonunique best rank-constrained approximations and is therefore inherently
set-valued.  The canonical transition family and the active covariance
operator provide a common framework for both questions.

\begin{acks}[Acknowledgments]
The author used ChatGPT
(OpenAI) to assist with editorial review, bibliographic checking, code review,
and independent computational cross-checks.  The author reviewed and verified
all AI-assisted suggestions and remains fully responsible for the mathematical
arguments, numerical results, references, originality, and integrity of the
manuscript.
\end{acks}

\clearpage
\setcounter{section}{0}
\setcounter{subsection}{0}
\setcounter{equation}{0}
\setcounter{theorem}{0}
\setcounter{table}{0}
\setcounter{figure}{0}
\renewcommand{\thesection}{S\arabic{section}}
\renewcommand{\thesubsection}{S\arabic{section}.\arabic{subsection}}
\renewcommand{\thetable}{S\arabic{table}}
\renewcommand{\thefigure}{S\arabic{figure}}
\makeatletter
\def\theHsection{supp.\arabic{section}}
\def\theHsubsection{supp.\arabic{section}.\arabic{subsection}}
\def\theHequation{supp.\arabic{section}.\arabic{equation}}
\def\theHtheorem{supp.\arabic{section}.\arabic{theorem}}
\def\theHtable{supp.\arabic{table}}
\def\theHfigure{supp.\arabic{figure}}
\makeatother

\begin{center}
{\LARGE\bfseries Supplement to ``Nonstandard likelihood-ratio limits under
semidefinite rank constraints''\par}
\vspace{1em}
{\large Didier Concordet\par}
\vspace{0.25em}
{\normalsize ORCID:
\href{https://orcid.org/0000-0003-3916-577X}{0000-0003-3916-577X}\par}
\vspace{1em}
\end{center}

\noindent\textbf{Supplementary abstract.}
This supplement contains expanded proofs and computational material for the
main article.  It gives the passage from uniform LAN to stratified Gaussian
suprema, the tangent cone and local set convergence for the PSD bounded-rank
null, the active-block quotient, and the null transition experiment.  It
also proves isotropic dominance for fixed strata and local rank transitions,
the anisotropic corank-one bound, and the conditional regular-sublevel lemma
supporting the positive-level Hadamard shape derivative.  Operator chain
rules and reproducible R calculations are recorded for the nuisance model,
rank transitions, anisotropic lower-stratum projections, the analytic
derivative check, and the Armijo orientation search.
\par\medskip

\noindent\textbf{Supplementary keywords:}
Constrained likelihood; low-rank semidefinite cone; probability
differentiation; supplementary proofs.

\section{From uniform LAN to stratified constrained suprema}
\label{sec:supp-LAN}

This section supplies the probabilistic part of the proof of
Theorem~\ref{thm:stratified-limit}, together with the nuisance-profile
calculation used throughout Section~\ref{sec:stratified}.  We work under the
framework of Section~\ref{sec:framework}, fix a true matrix of rank $s\le r$,
and retain the notation introduced there.

For a local direction $h=(a,H)\in\mathcal E$, $a\in\mathbb R^{d_\psi}$
is the nuisance direction and $H\in\Smat^q$ is the matrix direction.  Define the local log-likelihood
criterion
\begin{equation}
\mathbb L_n(h)=\ell_n(\theta_0+n^{-1/2}h)-\ell_n(\theta_0).
\label{eq:supp-local-criterion}
\end{equation}
Suppose that, under $\Pr_{\theta_0}$, this criterion satisfies
\begin{equation}
\mathbb L_n(h)
=
\langle h,Z_n\rangle
-
\frac12\langle h,\mathcal I h\rangle
+
\rho_n(h),
\label{eq:supp-LAN}
\end{equation}
where $\mathcal I$ is positive definite,
$Z_n\Rightarrow Z\sim\mathcal N_{\mathcal E}(0,\mathcal I)$, and
\begin{equation}
\sup_{h\in C}|\rho_n(h)|\xrightarrow{\Pr}0
\label{eq:supp-uniform-remainder}
\end{equation}
for every compact $C\subset\mathcal E$.

For $j\in\{0,1\}$, define the rescaled parameter sets
\begin{equation}
\mathcal H_{j,n}
=
\{h:\theta_0+n^{-1/2}h\in\Theta_j\},
\label{eq:supp-rescaled-sets}
\end{equation}
A closed set $\mathcal H_j$ is called their local limit when, for every
$R<\infty$,
\begin{equation}
d_{\mathrm H}(\mathcal H_{j,n}\cap\overline B_R,
\mathcal H_j\cap\overline B_R)\longrightarrow0,
\end{equation}
where $\overline B_R=\{h\in\mathcal E:\|h\|\le R\}$ is the closed ball in
$\mathcal E$.  To specify the application, let
\[
\mathcal M_r^+=\{\Sigma\in\Splus^q:\rank(\Sigma)\le r\},
\qquad K_s=\ker(\Sigma_0),\qquad k_s=q-s,
\]
let $P_s$ be the orthogonal projector onto $K_s$, choose an orthonormal
basis matrix $U_s\in\mathbb R^{q\times k_s}$ for $K_s$, and set
$\mathcal A_{P_s}(H)=U_s^\top HU_s$.  We use the abbreviations
\begin{equation}
T_{0,s}^{(r)}(P_s)=T_{\mathcal M_r^+}(\Sigma_0),
\qquad
T_1(P_s)=T_{\Splus^q}(\Sigma_0),
\label{eq:supp-tangent-abbreviations}
\end{equation}
whose explicit active-block forms are proved in
Section~\ref{sec:supp-tangent}.  The local sets are then
\begin{equation}
\mathcal H_0
=
\mathbb R^{d_\psi}\oplus T_{0,s}^{(r)}(P_s),
\qquad
\mathcal H_1
=
\mathbb R^{d_\psi}\oplus T_1(P_s).
\label{eq:supp-limit-sets}
\end{equation}

\subsection{A deterministic continuity lemma}

\begin{lemma}[Convergence of suprema on moving compact sets]
\label{lem:supp-moving-suprema}
Let $C_n,C$ be nonempty compact subsets of a finite-dimensional normed
space and suppose $d_{\mathrm H}(C_n,C)\to0$.  Let $f_n,f$ be continuous
and suppose $f_n\to f$ uniformly on one compact set containing $C$ and all
$C_n$ for large $n$.  Then
\begin{equation}
\sup_{x\in C_n}f_n(x)
\longrightarrow
\sup_{x\in C}f(x).
\end{equation}
\end{lemma}

\begin{proof}
Choose $x_n\in C_n$ satisfying
$f_n(x_n)\ge\sup_{C_n}f_n-n^{-1}$.  Every subsequence has a further
subsequence converging to a point $x\in C$, and uniform convergence gives
\begin{equation}
\limsup_n\sup_{C_n}f_n\le f(x)\le\sup_C f.
\end{equation}
Conversely, if $x^*$ maximises $f$ on $C$, Hausdorff convergence provides
$x_n^*\in C_n$ with $x_n^*\to x^*$.  Hence
\begin{equation}
\liminf_n\sup_{C_n}f_n
\ge
\lim_n f_n(x_n^*)
=
f(x^*)=
\sup_C f.
\end{equation}
\end{proof}

\subsection{Joint convergence of the constrained likelihood suprema}

For $R>0$, define
\begin{equation}
M_{j,n}(R)
=
\sup_{h\in\mathcal H_{j,n}\cap\overline B_R}
\{\ell_n(\theta_0+n^{-1/2}h)-\ell_n(\theta_0)\}.
\end{equation}
Let
\begin{equation}
M_j(R)
=
\sup_{h\in\mathcal H_j\cap\overline B_R}
\left\{
\langle h,Z\rangle-\frac12\langle h,\mathcal I h\rangle
\right\}.
\end{equation}

\begin{proposition}[Stratified Gaussian supremum limit]
\label{prop:supp-supremum-limit}
Assume \eqref{eq:supp-LAN}--\eqref{eq:supp-uniform-remainder}, local
Hausdorff convergence of $\mathcal H_{j,n}\cap\overline B_R$ to
$\mathcal H_j\cap\overline B_R$ for every $R$, and stochastic boundedness of the constrained local maximisers, meaning
that maximisers $\widehat h_{j,n}\in\mathcal H_{j,n}$ can be selected with
$\|\widehat h_{j,n}\|=O_{\Pr}(1)$ for $j=0,1$.  Then
\begin{equation}
\left(
\sup_{h\in\mathcal H_{0,n}}\mathbb L_n(h),
\sup_{h\in\mathcal H_{1,n}}\mathbb L_n(h)
\right)
\Rightarrow
\left(
\sup_{h\in\mathcal H_0}\mathbb G(h),
\sup_{h\in\mathcal H_1}\mathbb G(h)
\right),
\label{eq:supp-joint-suprema}
\end{equation}
where
\begin{equation}
\mathbb G(h)=\langle h,Z\rangle-\frac12\langle h,\mathcal I h\rangle.
\end{equation}
Consequently, twice the difference of the two likelihood suprema converges
to twice the difference of the Gaussian suprema.
\end{proposition}

\begin{proof}
Fix $\varepsilon>0$.  Stochastic boundedness yields $R<\infty$ such that,
with probability at least $1-\varepsilon$, both constrained maximisers lie
in $\overline B_R$.  On this event the unrestricted and truncated suprema
coincide.  On $\overline B_R$, put
\begin{equation}
R_{n,R}=\sup_{\|h\|\le R}|\rho_n(h)|.
\end{equation}
Then $R_{n,R}\to0$ in probability, and Slutsky's theorem gives the joint
convergence
\begin{equation}
(Z_n,R_{n,R})\Rightarrow(Z,0).
\end{equation}
Apply the Skorohod representation theorem to this pair, not to $Z_n$ alone.
On the representation space, $Z_n\to Z$ and $R_{n,R}\to0$ almost surely.
Consequently the complete LAN criteria, including their random remainders,
converge uniformly on $\overline B_R$.  Lemma~\ref{lem:supp-moving-suprema},
applied jointly to the two local sets, gives
\begin{equation}
(M_{0,n}(R),M_{1,n}(R))
\longrightarrow
(M_0(R),M_1(R))
\end{equation}
almost surely on the representation space.  Letting $R\to\infty$ removes
the truncation because the positive-definite quadratic term makes
$\mathbb G$ coercive and gives a unique unconstrained maximiser.  Finally
let $\varepsilon\downarrow0$.
\end{proof}

\subsection{Profiling the regular nuisance parameter}\label{subsec:supp-nuisance-profile}

Relative to $\mathcal E=\mathbb R^{d_\psi}\oplus\Smat^q$, write the
score and information operator in blocks:
\begin{equation}
Z=
\begin{pmatrix}Z_\psi\\Z_\Sigma\end{pmatrix},
\qquad
\mathcal I=
\begin{pmatrix}
\mathcal I_{\psi\psi}&\mathcal I_{\psi\Sigma}\\
\mathcal I_{\Sigma\psi}&\mathcal I_{\Sigma\Sigma}
\end{pmatrix}.
\end{equation}
Thus $\mathcal I_{\psi\Sigma}:\Smat^q\to\mathbb R^{d_\psi}$,
$\mathcal I_{\Sigma\psi}=\mathcal I_{\psi\Sigma}^*$, and the diagonal
blocks act on their corresponding component spaces.  Define
\begin{equation}
Z_\Sigma^{\mathrm{eff}}
=
Z_\Sigma-
\mathcal I_{\Sigma\psi}
\mathcal I_{\psi\psi}^{-1}Z_\psi,
\qquad
\mathcal I_{\mathrm{eff}}
=
\mathcal I_{\Sigma\Sigma}
-
\mathcal I_{\Sigma\psi}
\mathcal I_{\psi\psi}^{-1}
\mathcal I_{\psi\Sigma}.
\label{eq:supp-efficient-objects}
\end{equation}
\begin{lemma}[Schur-complement profile]
	\label{lem:supp-schur-profile}
	Let
	\begin{align}
		\mathbb G(a,H)
		&=
		\langle a,Z_\psi\rangle
		+
		\langle H,Z_\Sigma\rangle
		\nonumber\\
		&\quad
		-
		\frac12
		\left\{
		\langle a,\mathcal I_{\psi\psi}a\rangle
		+
		2\langle a,\mathcal I_{\psi\Sigma}H\rangle
		+
		\langle H,\mathcal I_{\Sigma\Sigma}H\rangle
		\right\}.
		\label{eq:supp-gaussian-criterion-blocks}
	\end{align}
	Then the operator
	\[
	\mathcal I_{\mathrm{eff}}
	=
	\mathcal I_{\Sigma\Sigma}
	-
	\mathcal I_{\Sigma\psi}
	\mathcal I_{\psi\psi}^{-1}
	\mathcal I_{\psi\Sigma}
	\]
	is positive definite, and
	\[
	Z_\Sigma^{\mathrm{eff}}
	=
	Z_\Sigma
	-
	\mathcal I_{\Sigma\psi}
	\mathcal I_{\psi\psi}^{-1}Z_\psi
	\sim
	\mathcal N(0,\mathcal I_{\mathrm{eff}}).
	\]
	For every fixed \(H\in\Smat^q\),
	\begin{equation}
		\sup_{a\in\mathbb R^{d_\psi}}\mathbb G(a,H)
		=
		C(Z_\psi)
		+
		\langle H,Z_\Sigma^{\mathrm{eff}}\rangle
		-
		\frac12
		\langle H,\mathcal I_{\mathrm{eff}}H\rangle,
		\label{eq:supp-profiled-criterion}
	\end{equation}
	where
	\begin{equation}
		C(Z_\psi)
		=
		\frac12
		\left\langle
		Z_\psi,
		\mathcal I_{\psi\psi}^{-1}Z_\psi
		\right\rangle
	\end{equation}
	is independent of \(H\).
	If
	\[
	G=\mathcal I_{\mathrm{eff}}^{-1}Z_\Sigma^{\mathrm{eff}},
	\]
	then the last two terms in
	\eqref{eq:supp-profiled-criterion} satisfy
	\begin{equation}
		\langle H,Z_\Sigma^{\mathrm{eff}}\rangle
		-
		\frac12
		\langle H,\mathcal I_{\mathrm{eff}}H\rangle
		=
		\frac12
		\langle G,\mathcal I_{\mathrm{eff}}G\rangle
		-
		\frac12
		\|H-G\|_{\mathcal I_{\mathrm{eff}}}^2.
	\end{equation}
\end{lemma}

\begin{proof}
	For fixed \(H\), the dependence of
	\(\mathbb G(a,H)\) on \(a\) is
	\begin{equation}
		\mathbb G(a,H)
		=
		\langle a,Z_\psi-\mathcal I_{\psi\Sigma}H\rangle
		-
		\frac12
		\langle a,\mathcal I_{\psi\psi}a\rangle
		+
		R(H),
	\end{equation}
	where
	\[
	R(H)
	=
	\langle H,Z_\Sigma\rangle
	-
	\frac12
	\langle H,\mathcal I_{\Sigma\Sigma}H\rangle
	\]
	does not depend on \(a\).
	Since \(\mathcal I_{\psi\psi}\) is positive definite, this is a
	strictly concave quadratic function of \(a\). Its unique maximiser is
	\begin{equation}
		a^*(H)
		=
		\mathcal I_{\psi\psi}^{-1}
		\bigl(
		Z_\psi-\mathcal I_{\psi\Sigma}H
		\bigr).
	\end{equation}
	Substitution gives
	\begin{align}
		\sup_a\mathbb G(a,H)
		&=
		\frac12
		\left\langle
		Z_\psi-\mathcal I_{\psi\Sigma}H,
		\mathcal I_{\psi\psi}^{-1}
		\bigl(
		Z_\psi-\mathcal I_{\psi\Sigma}H
		\bigr)
		\right\rangle
		\nonumber\\
		&\quad
		+
		\langle H,Z_\Sigma\rangle
		-
		\frac12
		\langle H,\mathcal I_{\Sigma\Sigma}H\rangle.
	\end{align}
	Expanding the quadratic term and using
	\(\mathcal I_{\Sigma\psi}=\mathcal I_{\psi\Sigma}^*\) yields
	\eqref{eq:supp-profiled-criterion}.
	
	The covariance identity
	\[
	\operatorname{Cov}
	\bigl(
	Z_\Sigma^{\mathrm{eff}}
	\bigr)
	=
	\mathcal I_{\mathrm{eff}}
	\]
	follows by direct block calculation. Positive definiteness of
	\(\mathcal I_{\mathrm{eff}}\) follows from the Schur-complement theorem,
	because \(\mathcal I\) is positive definite. The final identity follows by
	completing the square.
\end{proof}

\section{Tangent geometry of the bounded-rank PSD null}
\label{sec:supp-tangent}

Theorem~\ref{thm:supp-null-tangent} below is the expanded proof of Theorem~\ref{thm:null-tangent}; Proposition~\ref{prop:supp-local-Hausdorff} verifies the local-set part of Assumption~\ref{ass:local} for the direct semidefinite rank constraint.

Let
\begin{equation}
\mathcal M_r^+
=
\{\Sigma\in\Splus^q:\rank(\Sigma)\le r\}.
\end{equation}
Suppose $\Sigma_0\in\mathcal M_r^+$ has rank $s\le r$.  Let
$K_s=\ker(\Sigma_0)$, $k_s=q-s$, and let $P_s$ be the orthogonal projector
onto $K_s$.  Choose $U_s\in\mathbb R^{q\times k_s}$ with orthonormal
columns spanning $K_s$, so that $P_s=U_sU_s^\top$, and define
\begin{equation}
\mathcal A_{P_s}(H)=U_s^\top H U_s.
\end{equation}
For $m\le k$, set
\begin{equation}
\mathcal K_m^k
=
\{B\in\Splus^k:\rank(B)\le m\}.
\end{equation}

\subsection{Alternative tangent cone}

In an orthogonal basis adapted to
$\operatorname{range}(\Sigma_0)\oplus K_s$,
\begin{equation}
\Sigma_0=
\begin{pmatrix}\Lambda&0\\0&0\end{pmatrix},
\qquad \Lambda\in\Smat_{++}^s\ \text{(positive definite)},
\qquad
H=
\begin{pmatrix}A&B\\B^\top&C\end{pmatrix}.
\label{eq:supp-blocks}
\end{equation}
The classical PSD tangent formula is
\begin{equation}
T_{\Splus^q}(\Sigma_0)
=
\{H:C\succeq0\}
=
\{H:\mathcal A_{P_s}(H)\succeq0\}.
\label{eq:supp-PSD-tangent}
\end{equation}
Necessity follows from positivity of kernel quadratic forms.  For
sufficiency, one may use the factor path constructed below with no rank
restriction on $C$.

\subsection{Bounded-rank null tangent cone}\label{subsec:supp-null-tangent}

\begin{theorem}[Tangent cone to $\mathcal M_r^+$]
\label{thm:supp-null-tangent}
At a rank-$s$ point $\Sigma_0\in\mathcal M_r^+$,
\begin{equation}
T_{\mathcal M_r^+}(\Sigma_0)
=
\left\{
H:\mathcal A_{P_s}(H)\succeq0,
\ \rank\{\mathcal A_{P_s}(H)\}\le r-s
\right\}.
\label{eq:supp-null-tangent}
\end{equation}
\end{theorem}

\begin{proof}
Use the block representation \eqref{eq:supp-blocks}.  For necessity, let
$t_n\downarrow0$ and $H_n\to H$ be such that
$\Sigma_n=\Sigma_0+t_nH_n\in\mathcal M_r^+$.  For large $n$,
$\Lambda+t_nA_n$ is positive definite.  Its Schur complement is
\begin{equation}
R_n
=
t_nC_n-t_n^2B_n^\top
(\Lambda+t_nA_n)^{-1}B_n.
\end{equation}
Since $\Sigma_n\succeq0$, $R_n\succeq0$, and because the upper-left block
has rank $s$,
\begin{equation}
\rank(R_n)=\rank(\Sigma_n)-s\le r-s.
\end{equation}
After division by $t_n$,
\begin{equation}
C_n-t_nB_n^\top(\Lambda+t_nA_n)^{-1}B_n\to C.
\end{equation}
The set $\mathcal K_{r-s}^{k_s}$ is closed, so
$C\succeq0$ and $\rank(C)\le r-s$.

For sufficiency, suppose $C\succeq0$ and $\rank(C)=m\le r-s$.  Factor
$C=WW^\top$ with $W\in\mathbb R^{k_s\times m}$.  The Sylvester equation
\begin{equation}
\Lambda^{1/2}E^\top+E\Lambda^{1/2}=A
\end{equation}
has a solution because $\Lambda^{1/2}$ is positive definite.  Put
$F=B^\top\Lambda^{-1/2}$ and define
\begin{equation}
X(t)=
\begin{bmatrix}
\Lambda^{1/2}+tE&0\\
tF&\sqrt t\,W
\end{bmatrix}.
\label{eq:supp-factor-path}
\end{equation}
Then $\Sigma(t)=X(t)X(t)^\top\succeq0$ and
$\rank\{\Sigma(t)\}\le s+m\le r$.  Expansion gives
\begin{equation}
\Sigma(t)=\Sigma_0+tH+O(t^2),
\end{equation}
so $H$ is a Bouligand tangent direction.
\end{proof}

\begin{remark}[Top and lower strata]
When $s=r$, the active block must have rank zero and
$T_{\mathcal M_r^+}(\Sigma_0)=\{H:\mathcal A_{P_r}(H)=0\}$, the tangent
space of the rank-$r$ stratum.  When $s<r$, the active block ranges over the
nonconvex cone $\mathcal K_{r-s}^{q-s}$.
\end{remark}

\subsection{Local Hausdorff convergence}\label{subsec:supp-local-Hausdorff}

The pointwise tangent construction above can be made uniform without choosing
continuous matrix factors.

\begin{proposition}[Local Hausdorff convergence]
\label{prop:supp-local-Hausdorff}
For every $R<\infty$,
\begin{equation}
d_{\mathrm H}
\left[
\left\{t^{-1}(\Sigma-\Sigma_0):\Sigma\in\mathcal M_r^+\right\}
\cap\overline B_R,
T_{\mathcal M_r^+}(\Sigma_0)\cap\overline B_R
\right]
\longrightarrow0
\label{eq:supp-local-Hausdorff}
\end{equation}
as $t\downarrow0$.
\end{proposition}

\begin{proof}
For the outer inclusion, suppose by contradiction that there are
$t_n\downarrow0$ and
$H_n\in t_n^{-1}(\mathcal M_r^+-\Sigma_0)\cap\overline B_R$ whose
distance from the tangent cone is bounded below.  Compactness yields a
subsequence $H_n\to H$.  The necessity part of
Theorem~\ref{thm:supp-null-tangent} gives
$H\in T_{\mathcal M_r^+}(\Sigma_0)$, a contradiction.

For the inner inclusion, first fix $\varepsilon>0$ and contract the
direction.  For $H\in T_{\mathcal M_r^+}(\Sigma_0)\cap\overline B_R$, let
\begin{equation}
H_\varepsilon=(1-\eta)H,
\qquad
\eta=\frac{\varepsilon}{2(R+1)}.
\end{equation}
Then $\|H-H_\varepsilon\|_F\le\varepsilon/2$ and
$\|H_\varepsilon\|_F\le R-\eta R$ whenever $R>0$.  Use the block notation
\eqref{eq:supp-blocks} for $H_\varepsilon$.  Choose any factor
$C=WW^\top$.  Although $W$ need not vary continuously with $C$, there is
a finite constant $C_R$, depending only on $R$ and $\Sigma_0$, such that
$\|W\|_F^2=\tr(C)\le C_R$.  The solution $E$ of the Sylvester equation and
$F=B^\top\Lambda^{-1/2}$ satisfy bounds of the same form.  The symbol $C_R$
below denotes a finite constant with this dependence and may increase from
line to line.  The factor path has the exact expansion
\begin{equation}
X(t)X(t)^\top
=
\Sigma_0+tH_\varepsilon+t^2
\begin{pmatrix}
EE^\top&EF^\top\\
FE^\top&FF^\top
\end{pmatrix}.
\label{eq:supp-uniform-factor-expansion}
\end{equation}
Let $\Sigma_\varepsilon(t)=X(t)X(t)^\top$ for the factor built from
$H_\varepsilon$, and define the scaled feasible direction
\begin{equation}
\widetilde H_{t,\varepsilon}
=t^{-1}\{\Sigma_\varepsilon(t)-\Sigma_0\}.
\end{equation}
Then $\Sigma_0+t\widetilde H_{t,\varepsilon}\in\mathcal M_r^+$ and
\begin{equation}
\|\widetilde H_{t,\varepsilon}-H_\varepsilon\|_F\le C_Rt.
\end{equation}
For fixed $\varepsilon$, take $t$ sufficiently small that
$C_Rt\le\min(\varepsilon/2,\eta R)$; then
$\widetilde H_{t,\varepsilon}\in\overline B_R$ and
$\|\widetilde H_{t,\varepsilon}-H\|_F\le\varepsilon$.  The case $R=0$ is
trivial.  Since the choice is uniform over
$T_{\mathcal M_r^+}(\Sigma_0)\cap\overline B_R$, the inner excess tends to
zero.  Combining the two excess bounds proves
\eqref{eq:supp-local-Hausdorff}.
\end{proof}

The sets involved are semi-algebraic, so this direct result is also
consistent with the general Chernoff regularity theorem of
\citet{Drton2009}.  Smooth model embeddings require the corresponding
local coordinate map to transfer the uniform approximation; this is why the
main theorem retains local convergence as an explicit assumption.

\section{Active-block quotient, stratified limit, and local alternatives}
\label{sec:supp-active}

Together with Sections~\ref{sec:supp-LAN} and~\ref{sec:supp-tangent}, this section completes the proof of Theorem~\ref{thm:stratified-limit}.  Subsection~\ref{subsec:supp-local-transitions} proves Corollaries~\ref{cor:local-alternative} and~\ref{cor:null-transition}.

Recall the active compression $\mathcal A_{P_s}:\Smat^q\to\Smat^{k_s}$
from Section~\ref{sec:supp-tangent}.  Let $\mathcal I_{\mathrm{eff}}$ be a
self-adjoint positive-definite operator on $\Smat^q$ and let
$G\sim\mathcal N_{\Smat^q}(0,\mathcal I_{\mathrm{eff}}^{-1})$.  We use the
Fisher norm
\begin{equation}
\|H\|_{\mathcal I_{\mathrm{eff}}}^2
=\langle H,\mathcal I_{\mathrm{eff}}H\rangle_F.
\end{equation}
The adjoint of $\mathcal A_{P_s}$ with respect to the Frobenius inner
products is
\begin{equation}
\mathcal A_{P_s}^*(B)=U_sBU_s^\top.
\label{eq:supp-Astar}
\end{equation}
Define
\begin{equation}
\mathcal S_{P_s}
=
\mathcal A_{P_s}\mathcal I_{\mathrm{eff}}^{-1}
\mathcal A_{P_s}^*.
\label{eq:supp-SP}
\end{equation}

\subsection{Affine profiling}

\begin{lemma}[Active affine profile]
\label{lem:supp-affine-profile}
The operator $\mathcal S_{P_s}$ is self-adjoint and positive definite.  If
$B_s=\mathcal A_{P_s}(G)$, then, for every $B\in\Smat^{k_s}$,
\begin{equation}
\inf_{\mathcal A_{P_s}(H)=B}
\|H-G\|_{\mathcal I_{\mathrm{eff}}}^2
=
\left\langle
B-B_s,\mathcal S_{P_s}^{-1}(B-B_s)
\right\rangle.
\label{eq:supp-affine-value}
\end{equation}
The unique minimiser is
\begin{equation}
H_B
=
G+
\mathcal I_{\mathrm{eff}}^{-1}\mathcal A_{P_s}^*
\mathcal S_{P_s}^{-1}(B-B_s).
\label{eq:supp-HB}
\end{equation}
Moreover,
$B_s\sim\mathcal N_{\Smat^{k_s}}(0,\mathcal S_{P_s})$.
\end{lemma}

\begin{proof}
For nonzero $B$,
\begin{equation}
\langle B,\mathcal S_{P_s}B\rangle
=
\left\langle
\mathcal A_{P_s}^*B,
\mathcal I_{\mathrm{eff}}^{-1}\mathcal A_{P_s}^*B
\right\rangle>0,
\end{equation}
because $\mathcal A_{P_s}^*$ is injective.  Feasibility of
\eqref{eq:supp-HB} follows from the definition of $\mathcal S_{P_s}$.  If
$N\in\ker(\mathcal A_{P_s})$, then
\begin{equation}
\left\langle
N,\mathcal I_{\mathrm{eff}}(H_B-G)
\right\rangle
=
\left\langle
\mathcal A_{P_s}N,
\mathcal S_{P_s}^{-1}(B-B_s)
\right\rangle=0.
\end{equation}
Thus $H_B-G$ is metric-orthogonal to every feasible direction, proving
optimality.  Substitution gives \eqref{eq:supp-affine-value}.  The covariance
of $B_s$ is \eqref{eq:supp-SP}.
\end{proof}

\subsection{The stratified active-space statistic}\label{subsec:supp-active-statistic}

Let $m_s=r-s$ and define
\begin{equation}
\mathcal C_{P_s}
=
\mathcal S_{P_s}^{-1/2}(\Splus^{k_s}),
\qquad
\mathcal D_{P_s,m_s}
=
\mathcal S_{P_s}^{-1/2}(\mathcal K_{m_s}^{k_s}).
\label{eq:supp-CD}
\end{equation}
Combining Proposition~\ref{prop:supp-supremum-limit},
Theorem~\ref{thm:supp-null-tangent}, and
Lemma~\ref{lem:supp-affine-profile} gives
\begin{equation}
\Lambda_n\Rightarrow
\Delta_{s,r}(P_s)
=
\dist_F^2(Y,\mathcal D_{P_s,m_s})
-
\dist_F^2(Y,\mathcal C_{P_s}),
\label{eq:supp-stratified-limit}
\end{equation}
where $Y=\mathcal S_{P_s}^{-1/2}B_s$ is standard Gaussian in
$\Smat^{k_s}$.

\subsection{Basis invariance and sufficient operator equivalence}

If $\widetilde U_s=U_sO$ with
$O\in\operatorname O(k_s)$, let
$\widetilde{\mathcal A}_{P_s}(H)=\widetilde U_s^\top H\widetilde U_s$ and
\[
\widetilde{\mathcal S}_{P_s}
=\widetilde{\mathcal A}_{P_s}\mathcal I_{\mathrm{eff}}^{-1}
\widetilde{\mathcal A}_{P_s}^*.
\]
Define the orthogonal conjugation operator
\[
\mathcal Q_O:\Smat^{k_s}\longrightarrow\Smat^{k_s},
\qquad
\mathcal Q_O(B)=O^\top BO.
\]
Then
\begin{equation}
	\widetilde{\mathcal S}_{P_s}
	=
	\mathcal Q_O\mathcal S_{P_s}\mathcal Q_O^*.
\end{equation}
Moreover,
\[
\mathcal Q_O(\Splus^{k_s})=\Splus^{k_s},
\qquad
\mathcal Q_O(\mathcal K_{m_s}^{k_s})
=
\mathcal K_{m_s}^{k_s}.
\]
Consequently, the two whitened sets defined in
\eqref{eq:supp-CD} are transformed by the same orthogonal map, and the
standard Gaussian law is invariant under this transformation. Hence the
law in \eqref{eq:supp-stratified-limit} is independent of the chosen
orthonormal basis of $K_s$.

We next record a more general sufficient condition for two reduced
calibration problems to have the same limiting law. Let
$\mathcal S_1$ and $\mathcal S_2$ be self-adjoint positive-definite
operators on the same active space $\Smat^k$, and fix
$0\leq m\leq k$. For $i\in\{1,2\}$, define
\begin{equation}
	\mathcal C_i
	=
	\mathcal S_i^{-1/2}(\Smat^k_+),
	\qquad
	\mathcal D_{i,m}
	=
	\mathcal S_i^{-1/2}(\mathcal K_m^k).
\end{equation}
Suppose that there exist $a>0$ and an orthogonal operator
$\mathcal Q:\Smat^k\to\Smat^k$ such that
\begin{equation}
	\mathcal S_2
	=
	a\,\mathcal Q\mathcal S_1\mathcal Q^*,
	\label{eq:supp-projective-equivalence}
\end{equation}
and
\begin{equation}
	\mathcal Q(\Splus^k)=\Splus^k,
	\qquad
	\mathcal Q(\mathcal K_m^k)=\mathcal K_m^k.
	\label{eq:supp-cone-preservation}
\end{equation}
Since $\mathcal Q$ is orthogonal,
\[
\mathcal S_2^{-1/2}
=
a^{-1/2}\mathcal Q\mathcal S_1^{-1/2}\mathcal Q^*.
\]
Using \eqref{eq:supp-cone-preservation} and the fact that multiplication
of a cone by a positive scalar leaves it unchanged, we obtain
\[
\mathcal C_2=\mathcal Q\mathcal C_1,
\qquad
\mathcal D_{2,m}=\mathcal Q\mathcal D_{1,m}.
\]
Thus, if $Y$ is standard Gaussian in $\Smat^k$,
\[
\dist_F^2(Y,\mathcal D_{2,m})
-
\dist_F^2(Y,\mathcal C_2)
\stackrel{d}{=}
\dist_F^2(Y,\mathcal D_{1,m})
-
\dist_F^2(Y,\mathcal C_1).
\]
Therefore \eqref{eq:supp-projective-equivalence}, together with
\eqref{eq:supp-cone-preservation}, is a sufficient condition for equal
stratified calibration.

This condition is not necessary. Equality of the limiting distributions
requires only equality in law of the corresponding Gaussian
distance-difference functionals.

\subsection{Contiguous alternatives and null transition paths}\label{subsec:supp-local-transitions}

Consider the contiguous sequence
\begin{equation}
\theta_n=\theta_0+n^{-1/2}(a,H),
\qquad a\in\mathbb R^{d_\psi},\quad H\in\Smat^q,
\end{equation}
inside the ambient regular model.  Here $H$ is the matrix component of the
local drift; after nuisance profiling it is the efficient matrix drift.
Under $\Pr_{\theta_n}$, Le Cam's third lemma replaces $B_s$ by a Gaussian
variable with mean $\mathcal A_{P_s}(H)$ and covariance
$\mathcal S_{P_s}$.  Hence the
standard Gaussian in \eqref{eq:supp-stratified-limit} is shifted to
\begin{equation}
Y+\mu_s,
\qquad
\mu_s=\mathcal S_{P_s}^{-1/2}\mathcal A_{P_s}(H).
\label{eq:supp-local-drift}
\end{equation}
The same active sets determine local power.  Regular nuisance drifts cancel
through the efficient-score construction.

For a local null transition, let $t\mapsto\theta(t)\in\Theta_0$ be a
one-sided differentiable path satisfying
$\theta(t)=\theta_0+t(a,H)+o(t)$ as $t\downarrow0$, and set
$\theta_n=\theta(n^{-1/2})$.
The likelihood ratios of $\Pr_{\theta_n}$ with respect to
$\Pr_{\theta_0}$ have the usual LAN limit, so Le Cam's third lemma yields
the same shifted score.  The constrained parameter sets are still those
centred at $\theta_0$; their local limits are therefore the two tangent sets
already used in \eqref{eq:supp-stratified-limit}.  This proves the null
transition limit in the main article.

Because the path remains in the null, its derivative satisfies
$\mathcal A_{P_s}(H)\in\mathcal K_{m_s}^{k_s}$ by the necessity part of the
tangent-cone theorem.  Conversely, for any
$C\in\mathcal K_{m_s}^{k_s}$, the factor construction in
\eqref{eq:supp-factor-path} gives a null path with active derivative $C$.
Thus no admissible active transition drift is omitted.

\section{Isotropic strata, spectral projections, and computation}
\label{sec:supp-isotropic}

Subsection~\ref{subsec:supp-spectral-projection} proves Proposition~\ref{prop:isotropic-spectral}; Subsections~\ref{subsec:supp-fixed-ordering} and~\ref{subsec:supp-isotropic-transition} prove Theorem~\ref{thm:isotropic-dominance} and Corollary~\ref{cor:isotropic-transition-calibration}; Subsection~\ref{subsec:supp-corank-one} proves Proposition~\ref{prop:corank-one-transition}.

Assume that, after an irrelevant positive scaling,
$\mathcal S_{P_s}=\Id$ on $\Smat^{k_s}$ for every stratum.  Throughout this
section, $Y$ denotes a standard Gaussian matrix in the active symmetric
matrix space, $\lambda_j(X)$ is its $j$th largest eigenvalue when applied to
a symmetric matrix $X$, and $x_+=\max(x,0)$.  Then
$\mathcal C_{P_s}=\Splus^{k_s}$ and
$\mathcal D_{P_s,m_s}=\mathcal K_{m_s}^{k_s}$.

\subsection{Projection onto a rank-constrained PSD cone}\label{subsec:supp-spectral-projection}

Let
$X=V\diag(\lambda_1,\ldots,\lambda_k)V^\top$ be a spectral decomposition,
where $V\in\operatorname O(k)$ and
$\lambda_1\ge\cdots\ge\lambda_k$.  The Frobenius projection onto
$\Splus^k$ replaces negative eigenvalues by zero, while the projection onto
$\mathcal K_m^k$ retains only the $m$ largest positive eigenvalues.  Hence
\begin{equation}
\dist_F^2(X,\mathcal K_m^k)-\dist_F^2(X,\Splus^k)
=
\sum_{j=m+1}^k(\lambda_j^+)^2.
\label{eq:supp-isotropic-stat}
\end{equation}
The formula applies equally to $X=Y$ and to $X=Y+C$ for any fixed
symmetric shift $C\in\Smat^k$.

\subsection{Fixed-stratum ordering through interlacing}\label{subsec:supp-fixed-ordering}

Couple standard Gaussian matrices $Y_1,\ldots,Y_q$ so that $Y_k$ is the
leading $k\times k$ principal submatrix of $Y_{k+1}$.  If
$\lambda_1^{(k)}\ge\cdots\ge\lambda_k^{(k)}$ are the eigenvalues of $Y_k$,
Cauchy interlacing gives
\begin{equation}
\lambda_j^{(k+1)}\ge\lambda_j^{(k)}\ge\lambda_{j+1}^{(k+1)}.
\label{eq:supp-interlacing}
\end{equation}
For
\begin{equation}
\Delta_s
=
\sum_{j=r-s+1}^{q-s}\{\lambda_j^{(q-s)+}\}^2,
\end{equation}
the second inequality in \eqref{eq:supp-interlacing} pairs every term of
$\Delta_s$ with a larger term of $\Delta_{s+1}$.  Therefore
\begin{equation}
\Delta_0\le\Delta_1\le\cdots\le\Delta_r
\qquad\hbox{almost surely}.
\label{eq:supp-pathwise-dominance}
\end{equation}
This proves the fixed-stratum part of
the top-stratum dominance theorem in the main article.

\subsection{Isotropic null-transition dominance}\label{subsec:supp-isotropic-transition}

Let $k=q-s$, $m=r-s$, and $p=k-m=q-r$.  For an admissible active null drift
$C\in\mathcal K_m^k$ (equivalently, $C\succeq0$ and $\rank(C)\le m$), choose a $p$-dimensional subspace
$L\subseteq\ker(C)$ and let $V\in\mathbb R^{k\times p}$ have orthonormal
columns spanning $L$.  The Poincar\'e separation inequalities give
\begin{equation}
\lambda_{m+j}(Y+C)
\le
\lambda_j\{V^\top(Y+C)V\}
=
\lambda_j(V^\top YV),
\qquad j=1,\ldots,p.
\label{eq:supp-transition-separation}
\end{equation}
Because $Y$ is isotropic Gaussian, $V^\top YV$ is standard Gaussian in
$\Smat^p$.  Applying the increasing map $x\mapsto(x_+)^2$ and summing gives
\begin{equation}
\sum_{j=m+1}^{k}\{\lambda_j(Y+C)_+\}^2
\le
\sum_{j=1}^{p}\{\lambda_j(V^\top YV)_+\}^2.
\label{eq:supp-transition-dominance}
\end{equation}
The right-hand side is exactly the rank-$r$ top-stratum statistic.  This
proves the transition part of the top-stratum dominance theorem in the main article.

\subsection{The anisotropic active-corank-one case}\label{subsec:supp-corank-one}

Let $m=k-1$, let $\mathcal S$ be a self-adjoint positive-definite
operator on $\Smat^k$, and put
\begin{equation}
\mathcal C=\mathcal S^{-1/2}(\Splus^k),
\qquad
\mathcal D=\mathcal S^{-1/2}(\mathcal K_{k-1}^k).
\end{equation}
Since $\mathcal K_{k-1}^k=\partial\Splus^k$ and an invertible linear map is
a homeomorphism,
\begin{equation}
\mathcal D=\partial\mathcal C.
\label{eq:supp-boundary-image}
\end{equation}
Let $Y$ be standard Gaussian in $\Smat^k$.  Fix a drift
$\mu\in\partial\mathcal C$ and choose a unit outward supporting normal
$n\in\Smat^k$ at $\mu$, meaning that $\|n\|_F=1$ and
$\langle n,z-\mu\rangle_F\le0$ for every $z\in\mathcal C$.

For $x=\mu+Y$ outside $\mathcal C$, every nearest point in the closed convex
cone lies on its boundary, and hence
\begin{equation}
\dist(x,\partial\mathcal C)=\dist(x,\mathcal C).
\end{equation}
The likelihood-ratio distance difference is then zero.  If $x\in\mathcal C$,
move from $x$ in the outward direction $n$ until first reaching the boundary.
The supporting inequality shows that the required displacement is no larger
than $-\langle n,x-\mu\rangle=-\langle n,Y\rangle$.  Consequently,
\begin{equation}
\dist^2(x,\partial\mathcal C)-\dist^2(x,\mathcal C)
\le
\{-\langle n,Y\rangle_+\}^2.
\label{eq:supp-corank-one-bound}
\end{equation}
The scalar $-\langle n,Y\rangle$ is standard normal.  This proves the active-corank-one transition proposition in the main article.

\subsection{Numerical check for \texorpdfstring{$q=3$ and $r=2$}{q=3 and r=2}}

A simulation with $10^6$ independent standard Gaussian symmetric matrices
produces the following empirical quantiles:
\begin{table}[ht]
\caption{Isotropic fixed-stratum limits for $q=3$ and $r=2$.}
\label{tab:supp-isotropic}
\centering
\begin{tabular}{rrrrr}
\toprule
$s$ & $k_s$ & $m_s$ & empirical $0.90$ quantile & empirical $0.95$ quantile\\
\midrule
0&3&2&0.0000&0.0000\\
1&2&1&0.0356&0.2304\\
2&1&0&1.6424&2.7055\\
\bottomrule
\end{tabular}
\end{table}
The ordering is not an artefact of simulation; it is the pathwise inequality
\eqref{eq:supp-pathwise-dominance}.

\subsection{Anisotropic lower-stratum computation}

For general $\mathcal S_{P_s}$, projection onto
$\mathcal D_{P_s,m_s}$ is nonconvex.  It can be written as
\begin{equation}
\min_{B\succeq0,\,\rank(B)\le m_s}
\|\mathcal S_{P_s}^{-1/2}B-Y\|_F^2.
\label{eq:supp-nonconvex-projection}
\end{equation}
A factorisation $B=XX^\top$ with
$X\in\mathbb R^{k_s\times m_s}$ gives a smooth nonconvex least-squares
problem.  It is invariant under the right action $X\mapsto XQ$ of
$Q\in\operatorname O(m_s)$.  The numerical example in the main article uses the smallest nontrivial
case $q=2$, $r=1$, $s=0$.  For
$B=\begin{psmallmatrix}b_{11}&b_{12}\\b_{12}&b_{22}\end{psmallmatrix}$,
define the Lorentz-coordinate isometry
\begin{equation}
(u,v,w)=
\left(
\frac{b_{11}+b_{22}}{\sqrt2},
\frac{b_{11}-b_{22}}{\sqrt2},
\sqrt2\,b_{12}
\right).
\label{eq:supp-Lorentz-coordinates}
\end{equation}
It satisfies $\|B\|_F^2=u^2+v^2+w^2$ and
$B\succeq0$ if and only if $u\ge(v^2+w^2)^{1/2}$.  After the anisotropic
whitening used in the example, and with the relabelling
$(a,b,c)=(v,w,u)$, the alternative and null active sets are
\begin{equation}
\mathcal C_\gamma
=
\{(a,b,c):c\ge(a^2+\gamma b^2)^{1/2}\},
\qquad
\mathcal D_\gamma=\partial\mathcal C_\gamma,
\quad \gamma>0.
\label{eq:supp-lower-elliptic-cones}
\end{equation}
If $Y\notin\mathcal C_\gamma$, its metric projection onto the cone belongs
to the boundary, so the two distances in the likelihood-ratio statistic are
equal.  If $Y=(a,b,c)$ lies in the interior, the statistic is the squared
distance to $\partial\mathcal C_\gamma$.  At a generic point, a Lagrange
multiplier $\tau$ for the boundary equation satisfies
\begin{equation}
\frac{a^2}{(1-\tau)^2}
+
\frac{\gamma b^2}{(1-\gamma\tau)^2}
=
\frac{c^2}{(1+\tau)^2},
\qquad
0<\tau<\min(1,\gamma^{-1}).
\label{eq:supp-lower-aniso-tau}
\end{equation}
The left side minus the right side is strictly increasing on this interval,
so the generic root is unique and the projected point is
\begin{equation}
\left(
\frac{a}{1-\tau},
\frac{b}{1-\gamma\tau},
\frac{c}{1+\tau}
\right).
\label{eq:supp-lower-aniso-projection}
\end{equation}
The qualifier ``generic'' is necessary.  On exceptional coordinate axes the
minimum can occur at the closed endpoint and the nearest boundary point can
be nonunique, although its squared distance is unique.  If $0<\gamma<1$,
the endpoint $\tau=1$ can occur only when $a=0$ and
\begin{equation}
\frac{c^2}{4}\ge
\frac{\gamma b^2}{(1-\gamma)^2}.
\label{eq:supp-lower-aniso-endpoint-small-gamma}
\end{equation}
Then the projected point can be written
$(\widehat a,\widehat b,\widehat c)$ with
\begin{equation}
\widehat b=\frac{b}{1-\gamma},\qquad
\widehat c=\frac{c}{2},\qquad
\widehat a^2=\widehat c^2-\gamma\widehat b^2,
\end{equation}
with either sign of $\widehat a$ when the right-hand side is positive.  If
$\gamma>1$, the endpoint $\tau=\gamma^{-1}$ can occur only when $b=0$ and
\begin{equation}
\left(\frac{c}{1+\gamma^{-1}}\right)^2
\ge
\left(\frac{a}{1-\gamma^{-1}}\right)^2.
\label{eq:supp-lower-aniso-endpoint-large-gamma}
\end{equation}
Then the projected point can be written
$(\widehat a,\widehat b,\widehat c)$ with
\begin{equation}
\widehat a=\frac{a}{1-\gamma^{-1}},\qquad
\widehat c=\frac{c}{1+\gamma^{-1}},\qquad
\widehat b^2=\frac{\widehat c^2-\widehat a^2}{\gamma},
\end{equation}
again with either sign of $\widehat b$ when the right-hand side is positive.  These
exceptional axes have Gaussian probability zero but are handled explicitly in
the reproduction code.  For $\gamma=1$, the circular-cone identity reduces
to
\begin{equation}
\Delta_1^{\mathrm{low}}(a,b,c)
=
\frac12\{c-(a^2+b^2)^{1/2}\}^2
\mathbf1_{\{c>(a^2+b^2)^{1/2}\}}.
\end{equation}
Let $\Omega(\gamma)$ be the solid angle of $\mathcal C_\gamma$ and define
$v_3(\gamma)=\Omega(\gamma)/(4\pi)=\Pr\{Y\in\mathcal C_\gamma\}$ for
$Y\sim N_3(0,I_3)$.  The explicit integral for $\Omega(\gamma)$ is given in
\eqref{eq:supp-Omega}.  The atom at zero is therefore
$1-v_3(\gamma)$, providing a deterministic check on the Monte Carlo
implementation.  No conic Steiner representation is invoked for
this nonconvex null comparison set.

\section{Positive-level sensitivity on the top stratum}
\label{sec:supp-sensitivity}

Proposition~\ref{prop:supp-value-diff} is the detailed proof of Proposition~\ref{prop:value-sensitivity}.  Lemma~\ref{lem:supp-regular-sublevel} and Proposition~\ref{prop:supp-Hadamard-probability} provide the detailed proof of Theorem~\ref{thm:quantile-sensitivity} under Assumption~\ref{ass:surface-domination}.

Let $\mathcal S$ be a self-adjoint positive-definite operator on
$\Smat^k$, set $\mathcal R=\mathcal S^{-1/2}$, and let $Y\in\Smat^k$.
For an operator-valued argument, $\mathrm dV(\mathcal R,Y)[D]$ denotes the
Fr\'echet derivative in the operator direction $D$.  Define
\begin{equation}
V(\mathcal R,Y)
=
\frac12\min_{B\succeq0}\|\mathcal RB-Y\|_F^2,
\qquad
T(\mathcal R,Y)=\|Y\|_F^2-2V(\mathcal R,Y).
\label{eq:supp-value}
\end{equation}

\subsection{Continuity of the minimiser and envelope differentiation}\label{subsec:supp-value-differentiation}

\begin{lemma}[Uniform coercivity]
\label{lem:supp-coercivity}
Let $\mathcal U$ be an operator-norm compact set of invertible linear
operators on $\Smat^k$.
There exists $c>0$ such that
\begin{equation}
\|\mathcal RB\|_F\ge c\|B\|_F
\end{equation}
for every $\mathcal R\in\mathcal U$.  Consequently, minimisers of
\eqref{eq:supp-value} remain in a common compact set when
$(\mathcal R,Y)$ ranges over a compact set.
\end{lemma}

\begin{proof}
The smallest singular value is continuous and strictly positive on
$\mathcal U$, so its minimum is positive.  If
$\|Y\|_F\le M$, comparison with $B=0$ and the reverse triangle inequality
give
\begin{equation}
\frac12(c\|B\|_F-M)^2
\le
\frac12\|\mathcal RB-Y\|_F^2
\le
\frac12M^2
\end{equation}
at a minimiser, yielding a uniform bound on $\|B\|_F$.
\end{proof}

\begin{proposition}[Value differentiability]
\label{prop:supp-value-diff}
For every invertible $\mathcal R$ and $Y$, define
\[
B^*(\mathcal R,Y)=\argmin_{B\in\Splus^k}\|\mathcal RB-Y\|_F^2.
\]
This minimiser is unique and continuous in $(\mathcal R,Y)$.  The map
$\mathcal R\mapsto V(\mathcal R,Y)$ is continuously Fr\'echet differentiable
and, with $B^*=B^*(\mathcal R,Y)$,
\begin{equation}
\mathrm dV(\mathcal R,Y)[\dot{\mathcal R}]
=
\langle\mathcal RB^*-Y,\dot{\mathcal R}B^*\rangle.
\label{eq:supp-dV}
\end{equation}
Therefore
\begin{equation}
\mathrm dT(\mathcal R,Y)[\dot{\mathcal R}]
=
-2\langle\mathcal RB^*-Y,\dot{\mathcal R}B^*\rangle.
\label{eq:supp-dT}
\end{equation}
\end{proposition}

\begin{proof}
The Hessian of the objective with respect to $B$ is
$\mathcal R^*\mathcal R$, which is positive definite.  The objective is
therefore strongly convex on the convex set $\Splus^k$, giving uniqueness.
Lemma~\ref{lem:supp-coercivity} and the maximum theorem imply continuity of
the minimiser.  For a perturbation $\mathcal R+t\dot{\mathcal R}$, use
$B^*(\mathcal R,Y)$ as a feasible comparison to obtain the upper directional
bound, and use $B^*(\mathcal R+t\dot{\mathcal R},Y)$ in the reverse
comparison.  Continuity of the minimiser makes the two bounds coincide and
yields \eqref{eq:supp-dV}.  Continuity of the derivative follows from the
same ingredients.
\end{proof}

This argument does not require differentiability of the optimizer.  The
distinction between value differentiability and differentiability of the
metric projection is discussed by \citet{Shapiro2016}.  Strong regularity of
the Karush--Kuhn--Tucker (KKT) generalized equation is needed only if derivatives of $B^*$ or of the
dual multiplier are desired; see \citep{Robinson1980,ChanSun2008}.

\subsection{Automatic regularity of positive levels}

Let $C=\mathcal R(\Splus^k)$.  Its polar cone is
$C^\circ=\{Z\in\Smat^k:\langle Z,X\rangle_F\le0\text{ for all }X\in C\}$,
and $\Pi_C$ and $\Pi_{C^\circ}$ denote the Frobenius metric projections.
Moreau's decomposition gives
\begin{equation}
T(\mathcal R,Y)
=
\|\Pi_C(Y)\|_F^2
=
\dist_F^2(Y,C^\circ).
\label{eq:supp-polar-distance}
\end{equation}
Squared distance to a closed convex set is continuously differentiable, so
\begin{equation}
\nabla_YT(\mathcal R,Y)
=
2\{Y-\Pi_{C^\circ}(Y)\}
=
2\Pi_C(Y)
=
2\mathcal RB^*.
\label{eq:supp-Y-gradient}
\end{equation}
On $\{Y:T(\mathcal R,Y)=c\}$ with $c>0$,
\begin{equation}
\|\nabla_YT(\mathcal R,Y)\|_F=2\sqrt c.
\label{eq:supp-level-gradient}
\end{equation}
Thus every positive level is regular.  The only singular level relevant to
the chi-bar-square atom is $c=0$.

\subsection{A conditional regular-sublevel differentiation lemma}\label{subsec:supp-shape-lemma}

The following finite-dimensional lemma makes explicit which part of the
probability differentiation is geometric and which part is assumed through
the level functional.

\begin{lemma}[Regular-sublevel shape derivative]
\label{lem:supp-regular-sublevel}
Let $G_\vartheta:\mathbb R^d\to\mathbb R$ be indexed by
$\vartheta$ in a finite-dimensional normed vector space $\mathbb V$.  For
a direction $D\in\mathbb V$, write
$\dot G_\vartheta(y)[D]$ for the directional derivative of the map
$\vartheta\mapsto G_\vartheta(y)$.  Fix $(\vartheta_0,c)$ with $c>0$, a
truncation radius $M<\infty$, and some $\delta>0$.  Suppose that, on a
neighbourhood of the compact tube
\begin{equation}
\mathcal T_{M,\delta}
=\{y:\|y\|\le M,\ |G_{\vartheta_0}(y)-c|\le\delta\},
\end{equation}
(i) $G_\vartheta$ is continuously differentiable in $y$, jointly continuous
in $(\vartheta,y)$ together with $\nabla_yG_\vartheta$; (ii)
$\|\nabla_yG_\vartheta(y)\|$ is bounded away from zero for nearby
$\vartheta$; and (iii), for every $t_n\to0$ and $D_n\to D$,
\begin{equation}
\sup_{y\in\mathcal T_{M,\delta}}\left|
\frac{G_{\vartheta_0+t_nD_n}(y)-G_{\vartheta_0}(y)}{t_n}
-
\dot G_{\vartheta_0}(y)[D]
\right|\longrightarrow0.
\label{eq:supp-uniform-parameter-expansion}
\end{equation}
Let $w$ be continuously differentiable on the truncation ball, and let
$\mathcal H^{d-1}$ denote $(d-1)$-dimensional Hausdorff measure.  Define the
level functional
\begin{equation}
H_{\vartheta,D}^{(M)}(u)
=
\int_{\{G_\vartheta=u,\ \|y\|\le M\}}
\frac{\dot G_\vartheta(y)[D]}{\|\nabla_yG_\vartheta(y)\|}
 w(y)\,\mathrm d\mathcal H^{d-1}(y).
\end{equation}
If $(\vartheta,u,D)\mapsto H_{\vartheta,D}^{(M)}(u)$ is continuous at
$(\vartheta_0,c,D)$, uniformly for $D$ in bounded sets, then
\begin{equation}
\frac{
\int_{\|y\|\le M}\mathbf1\{G_{\vartheta_0+t_nD_n}(y)\le c\}w(y)\,\mathrm dy
-
\int_{\|y\|\le M}\mathbf1\{G_{\vartheta_0}(y)\le c\}w(y)\,\mathrm dy
}{t_n}
\longrightarrow
-H_{\vartheta_0,D}^{(M)}(c).
\label{eq:supp-sublevel-lemma-conclusion}
\end{equation}
\end{lemma}

\begin{proof}
The regular level set in the truncation ball is compact.  By the implicit
function theorem it admits a finite cover by coordinate charts in which one
coordinate, say $y_j$, is a continuously differentiable graph over the
remaining coordinates; the gradient lower bound and compactness make the
charts valid uniformly for nearby parameters.  Choose a continuously
differentiable partition of unity subordinate to this cover.  In one chart,
the boundary graph $b_\vartheta(z)$ satisfies
$G_\vartheta\{z,b_\vartheta(z)\}=c$.  The uniform expansion
\eqref{eq:supp-uniform-parameter-expansion} gives
\begin{equation}
\frac{b_{\vartheta_0+t_nD_n}(z)-b_{\vartheta_0}(z)}{t_n}
\longrightarrow
-
\frac{\dot G_{\vartheta_0}\{z,b_{\vartheta_0}(z)\}[D]}
{\partial_jG_{\vartheta_0}\{z,b_{\vartheta_0}(z)\}},
\end{equation}
uniformly on compact chart domains.  Differentiating the integral with its
moving endpoint and summing the partition-of-unity contributions yields
\eqref{eq:supp-sublevel-lemma-conclusion}; the graph area identity converts
$|\partial_jG|^{-1}$ into
$\|\nabla G\|^{-1}\,\mathrm d\mathcal H^{d-1}$.  Contributions outside the
finite cover vanish for sufficiently small perturbations because their
$G_{\vartheta_0}$ values are separated from $c$.  Continuity of the level
functional identifies the prescribed level and completes the proof.
\end{proof}

\subsection{Application to the Gaussian projection statistic}\label{subsec:supp-probability-derivative}

Identify $\Smat^k$ with $\mathbb R^d$, where $d=k(k+1)/2$, through the
Frobenius-orthonormal coordinates, and let
$\phi(Y)=(2\pi)^{-d/2}\exp(-\|Y\|_F^2/2)$ be the standard Gaussian density.
Define the distribution function
\begin{equation}
F_{\mathcal R}(c)=
\Pr\{T(\mathcal R,Y)\le c\}
=\int_{\Smat^k}\mathbf1\{T(\mathcal R,Y)\le c\}\phi(Y)\,\mathrm dY,
\end{equation}
where $Y$ is standard Gaussian in $\Smat^k$.  For $\alpha\in(0,1)$, set
\begin{equation}
q_\alpha(\mathcal R)=\inf\{c:F_{\mathcal R}(c)\ge1-\alpha\}.
\end{equation}
When it exists at $c>0$, $f_{\mathcal R}(c)$ denotes the density of the
continuous component of this law.  For a self-adjoint operator direction
$D$ on $\Smat^k$, define
\begin{equation}
h_{\widetilde{\mathcal R},D}(t)
=
\int_{\{T(\widetilde{\mathcal R},Y)=t\}}
\frac{\mathrm dT(\widetilde{\mathcal R},Y)[D]}
{\|\nabla_YT(\widetilde{\mathcal R},Y)\|_F}
\phi(Y)\,\mathrm d\mathcal H^{d-1}(Y).
\label{eq:supp-level-functional}
\end{equation}
Assumption~\ref{ass:surface-domination} of the main article is the continuity and tail condition
needed to pass from the truncated lemma to the Gaussian integral.

\begin{proposition}[Conditional Hadamard probability derivative]
\label{prop:supp-Hadamard-probability}
Under Assumption~\ref{ass:surface-domination} of the main article, for every $t_n\to0$ and every
self-adjoint $D_n\to D$ such that $\mathcal R+t_nD_n$ remains invertible,
\begin{equation}
\frac{F_{\mathcal R+t_nD_n}(c)-F_{\mathcal R}(c)}{t_n}
\longrightarrow
-h_{\mathcal R,D}(c).
\label{eq:supp-Hadamard-probability}
\end{equation}
\end{proposition}

\begin{proof}
On each truncation ball, Proposition~\ref{prop:supp-value-diff} gives the
uniform parameter expansion in
\eqref{eq:supp-uniform-parameter-expansion}.  Joint continuity of the
minimiser gives joint continuity of $T$ and its operator derivative.
Equation \eqref{eq:supp-Y-gradient} gives joint continuity of the Gaussian
gradient, and \eqref{eq:supp-level-gradient} bounds its norm away from zero
on every compact positive-level tube.  Lemma~\ref{lem:supp-regular-sublevel}
therefore applies with $G_{\mathcal R}=T(\mathcal R,\cdot)$ and
$w=\phi$.  The tail condition in Assumption~\ref{ass:surface-domination} lets the truncation radius
tend to infinity uniformly over bounded direction sequences, and its level
continuity evaluates the limit at the prescribed $c$.
\end{proof}

Write $\mathrm d_H$ for the Hadamard directional derivative with respect
to the operator argument, namely the common limit along all
$t_n\to0$ and $D_n\to D$ for which $\mathcal R+t_nD_n$ remains in the
operator domain.  Because $\|\nabla_YT\|_F=2\sqrt c$ on the target level,
\begin{equation}
\mathrm d_HF_{\mathcal R}(c)[D]
=
-\frac1{2\sqrt c}
\int_{\{Y:T(\mathcal R,Y)=c\}}\mathrm dT(\mathcal R,Y)[D]\phi(Y)
\,\mathrm d\mathcal H^{d-1}(Y).
\label{eq:supp-dF}
\end{equation}
If $q_\alpha(\mathcal R)>0$, the continuous component has a density that is
continuous and positive at the target quantile, and Assumption~\ref{ass:surface-domination} holds on
a neighbourhood of that level, the Hadamard implicit-function theorem gives
\begin{equation}
\mathrm d_Hq_\alpha(\mathcal R)[D]
=
-\frac{\mathrm d_HF_{\mathcal R}\{q_\alpha(\mathcal R)\}[D]}
{f_{\mathcal R}\{q_\alpha(\mathcal R)\}}.
\label{eq:supp-dq}
\end{equation}
For the explicit elliptic-cone family in the main article,
$F_\gamma(c)=\sum_{j=0}^3v_j(\gamma)F_{\chi_j^2}(c)$, where
$F_{\chi_j^2}$ is the chi-square distribution function with $j$ degrees of
freedom and $\chi_0^2$ is the atom at zero.  The weights are
continuously differentiable for $\gamma>0$.  Thus the derivative used
numerically is verified independently of the general level-set assumption;
this family-specific check does not establish Assumption~\ref{ass:surface-domination} for arbitrary
transformed PSD cones.

\section{Operator and Grassmannian derivatives}
\label{sec:supp-derivatives}

\subsection{Derivative of the active compression}

Let $\Gr(k,q)$ denote the Grassmann manifold represented as the set of
rank-$k$ orthogonal projectors $P\in\Smat^q$.  Write $P=UU^\top$ with
$U\in\mathbb R^{q\times k}$ satisfying $U^\top U=I_k$, where $I_k$
is the $k\times k$ identity matrix, and define
$\mathcal A_P(H)=U^\top HU$ and $\mathcal A_P^*(B)=UBU^\top$.  For a smooth
curve $P(t)\in\Gr(k,q)$ with $P(0)=P$, a dot denotes differentiation at
$t=0$.  A tangent direction $\dot P\in T_P\Gr(k,q)$ can be represented by
the horizontal lift
\begin{equation}
\dot U=\dot P U,
\qquad U^\top\dot U=0.
\end{equation}
Then
\begin{equation}
\dot{\mathcal A}_P(H)
=
\dot U^\top HU+U^\top H\dot U,
\label{eq:supp-dA}
\end{equation}
and
\begin{equation}
\dot{\mathcal A}_P^*(B)
=
\dot UBU^\top+UB\dot U^\top.
\label{eq:supp-dAstar}
\end{equation}
Define
\begin{equation}
\mathcal S_P=\mathcal A_P\mathcal I_{\mathrm{eff}}^{-1}\mathcal A_P^*.
\end{equation}
At fixed $\mathcal I_{\mathrm{eff}}$,
\begin{equation}
\dot{\mathcal S}_P
=
\dot{\mathcal A}_P\mathcal I_{\mathrm{eff}}^{-1}\mathcal A_P^*
+
\mathcal A_P\mathcal I_{\mathrm{eff}}^{-1}\dot{\mathcal A}_P^*.
\label{eq:supp-dS}
\end{equation}
If $\mathcal I_{\mathrm{eff}}=\mathcal I_{\mathrm{eff}}(t)$ also varies
along the curve, write
$\dot{\mathcal I}_{\mathrm{eff}}=\mathrm d\mathcal I_{\mathrm{eff}}(t)/\mathrm dt|_{t=0}$
and add
\begin{equation}
-
\mathcal A_P\mathcal I_{\mathrm{eff}}^{-1}
\dot{\mathcal I}_{\mathrm{eff}}
\mathcal I_{\mathrm{eff}}^{-1}\mathcal A_P^*.
\label{eq:supp-dS-info}
\end{equation}

\subsection{Derivative of the inverse square root}

Let $\mathcal S$ be a self-adjoint positive-definite operator on a
finite-dimensional Euclidean space, and put
$\mathcal L=\mathcal S^{1/2}$ and $\mathcal R=\mathcal S^{-1/2}$.
The derivative $\dot{\mathcal L}$ is the unique self-adjoint solution of
\begin{equation}
\mathcal L\dot{\mathcal L}
+
\dot{\mathcal L}\mathcal L
=
\dot{\mathcal S}.
\label{eq:supp-Sylvester}
\end{equation}
Uniqueness follows because the spectra of $\mathcal L$ and $-\mathcal L$
are disjoint.  Differentiating
$\mathcal R=\mathcal L^{-1}$ gives
\begin{equation}
\dot{\mathcal R}
=
-\mathcal R\dot{\mathcal L}\mathcal R.
\label{eq:supp-dR}
\end{equation}
Equations \eqref{eq:supp-dS}--\eqref{eq:supp-dR}, composed with
\eqref{eq:supp-dq}, give the derivative with respect to the kernel
projector.

\subsection{Ambient and Riemannian gradients}

Let $\mathcal R(P)=\mathcal S_P^{-1/2}$ and let $q_\alpha(P)$ abbreviate
the upper-$\alpha$ critical value $q_\alpha\{\mathcal R(P)\}$ obtained by
composing the active operator
with the quantile map of Section~\ref{sec:supp-sensitivity}.  Suppose an
ambient symmetric matrix $\Gamma_P\in\Smat^q$ represents its differential:
\begin{equation}
\mathrm dq_\alpha(P)[\dot P]
=
\langle\Gamma_P,\dot P\rangle_F.
\end{equation}
The tangent space in the projector representation is
\[
T_P\Gr(k,q)=\{\dot P\in\Smat^q:P\dot PP=0,\ (I_q-P)\dot P(I_q-P)=0\},
\]
For matrices $A,B$, write $[A,B]=AB-BA$ for their commutator.  The
Frobenius-orthogonal projection of a symmetric ambient matrix onto
$T_P\Gr(k,q)$ is
\begin{equation}
\Pi_{T_P}(\Gamma)
=
[P,[P,\Gamma]]
=
(I_q-P)\Gamma P+P\Gamma(I_q-P).
\end{equation}
Hence
\begin{equation}
\operatorname{grad}q_\alpha(P)
=
[P,[P,\Gamma_P]].
\label{eq:supp-Grassmann-gradient}
\end{equation}
In a basis representation, the horizontal gradient is
\begin{equation}
G_U=2(I_q-UU^\top)\Gamma_PU.
\end{equation}
For a full-column-rank matrix $A$, let $\qf(A)$ denote the $Q$ factor of its
thin QR decomposition, with a fixed positive-diagonal convention for $R$.
A retracted ascent step is
\begin{equation}
U_+=\qf(U+\eta G_U),
\qquad P_+=U_+U_+^\top,
\end{equation}
where $\eta>0$ is the trial step length selected by backtracking.  Equations
\eqref{eq:supp-dS}--\eqref{eq:supp-dR} define the pullback through linear
maps and a Sylvester solve.  The remaining ambient functional is the surface
integral in \eqref{eq:supp-dF}; constructing a stable general Monte Carlo
estimator for that integral is outside the present scope.  The reported
one-dimensional example bypasses this issue because its derivative is
available analytically.

\section{Model calculations and reproducible numerical details}
\label{sec:supp-numerics}

Subsection~\ref{subsec:supp-Gaussian-information} proves Proposition~\ref{prop:Gaussian-Ieff}; the remaining subsections document the calculations used in Section~\ref{sec:applications}.

\subsection{Gaussian covariance information}\label{subsec:supp-Gaussian-information}

Let $V\in\Smat_{++}^m$, and let $X_1,\ldots,X_n$ be independent with
$X_i\sim\mathcal N_m(0,V)$, where $\mathcal N_m$ denotes the
$m$-variate normal law.  For a generic observation $X$ and a covariance
direction $A\in\Smat^m$, the score is
\begin{equation}
\dot\ell_A(X)
=
\frac12\left
\{X^\top V^{-1}AV^{-1}X-\tr(V^{-1}A)
\right
\}.
\end{equation}
The per-observation Fisher inner product is
\begin{equation}
\langle A,B\rangle_V
=
\frac12\tr(V^{-1}AV^{-1}B).
\end{equation}
Endow $\Smat^m$ with the Fisher inner product $\langle\cdot,\cdot\rangle_V$
above, the nuisance space $\mathbb R^{d_\psi}$ with its Euclidean inner
product, and $\Smat^q$ with the Frobenius inner product.  Let $R(\psi)\in\Smat_{++}^m$ be differentiable near $\psi_0$ and let
$L\in\mathbb R^{m\times q}$ be fixed.  For
$V(\psi,\Sigma)=R(\psi)+L\Sigma L^\top$, define the Fr\'echet derivative
operators
\[
\mathcal D_\psi a=\mathrm dR(\psi_0)[a],
\qquad
\mathcal D_\Sigma H=LHL^\top.
\]
Their adjoints below are taken with respect to these specified inner
products.  The information blocks are
\begin{equation}
\mathcal I_{\psi\psi}=\mathcal D_\psi^*\mathcal D_\psi,
\quad
\mathcal I_{\psi\Sigma}=\mathcal D_\psi^*\mathcal D_\Sigma,
\quad
\mathcal I_{\Sigma\Sigma}=\mathcal D_\Sigma^*\mathcal D_\Sigma.
\end{equation}
Assume that $\mathcal D_\psi$ has full column rank, so that
$\mathcal D_\psi^*\mathcal D_\psi$ is invertible.  Taking the Schur
complement gives
\begin{equation}
\mathcal I_{\mathrm{eff}}
=
\mathcal D_\Sigma^*(I-\Pi_\psi)\mathcal D_\Sigma,
\quad
\Pi_\psi
=
\mathcal D_\psi
(\mathcal D_\psi^*\mathcal D_\psi)^{-1}
\mathcal D_\psi^*,
\end{equation}
where $\Id$ is the identity operator on $\Smat^m$ and $\Pi_\psi$ is the
$\langle\cdot,\cdot\rangle_V$-orthogonal projector onto
$\operatorname{range}(\mathcal D_\psi)$.

\subsection{Residual-variance nuisance example}

For
\begin{equation}
V(\tau,\Sigma)=\tau I_3+J\Sigma J^\top,
\qquad J=(e_1,e_2),
\end{equation}
where $e_1,e_2,e_3$ are the standard basis vectors of $\mathbb R^3$ and
$J$ is the $3\times2$ matrix with columns $e_1,e_2$.  At
$(\tau,\Sigma)=(1,0)$, use the Frobenius-orthonormal basis
$E_{11},E_{22},F_{12}$ of $\Smat^2$, where
$F_{12}=(E_{12}+E_{21})/\sqrt2$.  The nuisance derivative is $I_3$,
and the matrix derivatives are the embedded basis elements.  Therefore
\begin{equation}
\mathcal I_{\tau\tau}=\frac12\tr(I_3)=\frac32,
\qquad
\mathcal I_{\Sigma\Sigma}=\frac12\Id,
\qquad
\mathcal I_{\Sigma\tau}
=
\begin{pmatrix}1/2\\1/2\\0\end{pmatrix}.
\end{equation}
The Schur complement is
\begin{equation}
\mathcal I_{\mathrm{eff}}
=
\begin{pmatrix}
1/3&-1/6&0\\
-1/6&1/3&0\\
0&0&1/2
\end{pmatrix}.
\end{equation}
In the Lorentz coordinates $(u,v,w)$ defined in
\eqref{eq:supp-Lorentz-coordinates}, it is
$\diag(1/6,1/2,1/2)$, so the active covariance is
$\diag(6,2,2)$.  The whitened cone has half-angle $\pi/3$.  For a circular cone in $\mathbb R^3$ with half-angle $\beta$, denote its
conic intrinsic-volume weights by $(v_0,v_1,v_2,v_3)$.  They satisfy
\begin{equation}
v_3=\frac{1-\cos\beta}{2},
\qquad
v_0=\frac{1-\sin\beta}{2},
\qquad
v_1=\frac12-v_3,
\qquad
v_2=\frac12-v_0.
\end{equation}
At $\beta=\pi/3$ this gives the weights reported in the main article.

For the finite-sample likelihood, let $d_1\ge d_2$ be the eigenvalues of
the upper $2\times2$ sample-covariance block and let $d_3$ be the third
coordinate sample variance.  For fixed $t=\tau$, maximisation over
$\Sigma\succeq0$ sets the two fitted upper eigenvalues equal to
$\max(d_j,t)$, $j=1,2$.  The profile
objective is therefore
\begin{equation}
g(t)
=
\log t+\frac{d_3}{t}
+
\sum_{j=1}^2
\left\{
\log\max(d_j,t)+\frac{d_j}{\max(d_j,t)}
\right\}.
\end{equation}
The minimum occurs among the stationary points in the intervals determined
by $d_1,d_2$ and the interval endpoints.  In the simulation, all candidates
were evaluated explicitly; no generic nonlinear solver was required.

The experiment used independent standard Gaussian samples, $10^6$
replications for each $n\in\{20,50,100,250,1000\}$, and the asymptotic
critical value $6.1252334478$.  If $\widehat p$ denotes the empirical rejection proportion, its Monte
Carlo standard error was computed as
$\{\widehat p(1-\widehat p)/10^6\}^{1/2}$.

\subsection{Anisotropic cone quadrature and derivative}

For
\begin{equation}
\mathcal C_\gamma
=
\{(u,v,w):u\ge(v^2+\gamma w^2)^{1/2}\},
\qquad \gamma>0,
\label{eq:supp-top-elliptic-cone}
\end{equation}
the solid angle is
\begin{equation}
\Omega(\gamma)
=
\int_0^{2\pi}
\left\{
1-
\frac{a_\gamma(\phi)}{\sqrt{1+a_\gamma(\phi)^2}}
\right\}\,\mathrm d\phi,
\quad
a_\gamma(\phi)=
\sqrt{\cos^2\phi+\gamma\sin^2\phi}.
\label{eq:supp-Omega}
\end{equation}
The conic intrinsic-volume weights and the corresponding chi-bar-square cumulative distribution function (CDF)
are
\begin{equation}
\begin{aligned}
v_3(\gamma)&=\frac{\Omega(\gamma)}{4\pi},
&v_0(\gamma)&=\frac{\Omega(\gamma^{-1})}{4\pi},\\
v_1(\gamma)&=\frac12-v_3(\gamma),
&v_2(\gamma)&=\frac12-v_0(\gamma),
\end{aligned}
\qquad
F_\gamma(c)=\sum_{j=0}^3v_j(\gamma)F_{\chi_j^2}(c).
\label{eq:supp-anisotropic-weights}
\end{equation}
Here $F_{\chi_j^2}$ is the chi-square CDF with $j$ degrees of freedom and
$F_{\chi_0^2}(c)=\mathbf1_{\{c\ge0\}}$.  The reported upper-$5\%$ critical
value is $q_{0.05}(\gamma)=\inf\{c:F_\gamma(c)\ge0.95\}$.
Adaptive Gauss--Kronrod quadrature with absolute and relative tolerances
$10^{-12}$ was used for $\Omega$ and $\Omega'$.  The chi-bar-square CDF was
inverted by a bracketing root finder with tolerance $10^{-12}$.
Differentiation under the one-dimensional integral gives
\begin{equation}
\Omega'(\gamma)
=
-\frac12
\int_0^{2\pi}
\frac{\sin^2\phi}{
a_\gamma(\phi)\{1+a_\gamma(\phi)^2\}^{3/2}}
\,\mathrm d\phi.
\end{equation}
The weight derivatives, followed by implicit differentiation of the mixture
CDF, produce $q_{0.05}'(1)=-0.339124730022$.  An independent central difference at $h=10^{-5}$ equals $-0.339124730075$, with absolute discrepancy $5.34\times10^{-11}$.

\subsection{Ascent protocol}

For
\begin{equation}
\gamma(\theta)=0.25\cos^2\theta+4\sin^2\theta,
\end{equation}
the one-dimensional gradient is
\begin{equation}
\frac{\mathrm d}{\mathrm d\theta}q_{0.05}\{\gamma(\theta)\}
=
q_{0.05}'\{\gamma(\theta)\}
(4-0.25)\sin(2\theta).
\end{equation}
Write
$g(\theta)=\mathrm d q_{0.05}\{\gamma(\theta)\}/\mathrm d\theta$.
At each iteration the algorithm proposes
$\theta_+=\theta+\eta g(\theta)$ and accepts the first
$\eta\in\{1,1/2,1/4,\ldots\}$ satisfying the Armijo condition
\[
q_{0.05}\{\gamma(\theta_+)\}
\ge q_{0.05}\{\gamma(\theta)\}+10^{-4}\eta g(\theta)^2.
\]
Starting from $\theta\in\{0.2,0.6,1.0,1.3\}$, the executed algorithm used
gradient tolerance $10^{-7}$ and minimum step $10^{-12}$.  The canonical
Grassmann angle reported below is
$\bar\theta=\min\{\theta\bmod\pi,\pi-(\theta\bmod\pi)\}\in[0,\pi/2]$.  Every run reached canonical angle below $2.6\times10^{-8}$ and
critical value $5.877087$.  The starts $0.2$ and $1.3$ met the gradient
tolerance; the starts $0.6$ and $1.0$ stopped when no further Armijo increase
was distinguishable above the minimum step, with final gradient magnitude
below $1.7\times10^{-7}$.  The complete trace and backtrack counts are
exported rather than being described as exact convergence.

\begin{remark}[Reproducibility]
The supplied file \texttt{code\_Geometry\_LR\_calibration.R} is the source of every
reported numerical value.  It records the exact critical value, analytic
derivative, Armijo parameters, random-number conventions, seeds, and
replication counts.  Section~\ref{sec:supp-reproducibility} gives the full
file-to-table map.
\end{remark}

\section{Additional remarks on stratified uniformity}
\label{sec:supp-uniformity}

Let
$\mathfrak N_s=\{\Sigma\in\Splus^q:\rank(\Sigma)=s\}$ be the rank-$s$
null stratum, and let $\mathfrak C_s\subset\mathfrak N_s$ be compact.  Assume
on $\mathfrak C_s$ that all nonzero eigenvalues are bounded away from zero,
the information is
uniformly positive definite, and the LAN and localisation conditions hold
uniformly.  Then the proof of
Proposition~\ref{prop:supp-supremum-limit} is uniform in the parameter, and
continuity of the active distance functional gives uniform convergence of
the LRT distribution on $\mathfrak C_s$.  Maximising the fixed-stratum
critical values therefore controls finite unions of such rank-separated
compact sets.

At a rank interface, a smallest positive eigenvalue of order $n^{-1/2}$
appears as a finite deterministic drift in the active experiment.  The null
transition corollary in the main article makes this statement precise for
every differentiable null path.  Under isotropy,
\eqref{eq:supp-transition-dominance} bounds every such translated law by the
top-stratum law.  The same conclusion follows from
\eqref{eq:supp-corank-one-bound} for arbitrary anisotropy when $q-r=1$.
Uniform versions over compact drift sets follow when the LAN remainder,
localisation, and third-lemma approximation are uniform over those sets.

Within the present framework, the unresolved case is anisotropic active
corank $q-r\ge2$.  There the null active set is only a proper subset of the
boundary of the alternative cone, and the rank-constrained translated
projection need not admit either the spectral compression or the
supporting-hyperplane bound used above.

\section{Reproducibility and additional transition experiment}
\label{sec:supp-reproducibility}

All numerical claims in the main article are produced by the supplied base-R
script
\begin{verbatim}
code_Geometry_LR_calibration.R
\end{verbatim}
The option \texttt{--quick} is a smoke test and was not used for reported
values.  A full run writes the following files:
\begin{enumerate}
\item \texttt{isotropic\_fixed\_strata.csv}: the three fixed-stratum laws
in Table~\ref{tab:strata-isotropic} for $q=3$, $r=2$;
\item \texttt{nuisance\_size.csv}: nuisance-model rejection probabilities
and empirical quantiles;
\item \texttt{lower\_transition.csv}: fixed lower-stratum and local
rank-transition experiments;
\item \texttt{isotropic\_transition\_dominance.csv}: coupled checks of the
isotropic pathwise transition bound;
\item \texttt{anisotropic\_weights.csv}: top-stratum intrinsic-volume
weights and critical values;
\item \texttt{derivative\_check.csv}: analytic derivative and the independent
$h=10^{-5}$ finite difference;
\item \texttt{finite\_differences.csv}: the full convergence table;
\item \texttt{ascent.csv}: endpoint, backtracking, and termination summaries;
\item \texttt{ascent\_trace.csv}: every accepted Armijo iterate;
\item \texttt{lower\_anisotropic.csv}: the nonconvex anisotropic lower-stratum
experiment;
\item {\footnotesize\texttt{corank\_one\_transition\_dominance.csv}}: supporting-hyperplane
checks for anisotropic active corank one;
\item \texttt{reproduction\_manifest.csv}: R version, RNG convention,
replication counts, analytic and Monte Carlo critical values, and derivative
checks;
\item \texttt{sessionInfo.txt}: complete R session information; and
\item \texttt{README.txt}: execution commands and the mapping from output
files to manuscript tables.
\end{enumerate}

In the following display, $q_{0.05}$ denotes an exact or deterministic
upper-$5\%$ critical value and $\widehat q_{0.05}$ an empirical Monte Carlo
quantile.  The full-run checks are
\begin{equation}
q_{0.05}^{\mathrm{lower}}=1.369475,
\qquad
\widehat q_{0.05}^{\mathrm{top,MC}}=5.490850,
\qquad
q_{0.05}^{\mathrm{top}}=q_{0.05}^{\mathrm{aniso}}(1)
=5.4845131865391.
\label{eq:supp-reproduction-checks}
\end{equation}
The seed-9002 Monte Carlo value is retained as a diagnostic.  The analytic
value is used for all top-stratum rejection decisions in the
revised script.  For a step $h>0$, define the centred finite-difference approximation
\begin{equation}
q_{0.05,h}^{\mathrm{FD}\prime}(1)
=\frac{q_{0.05}(1+h)-q_{0.05}(1-h)}{2h}.
\end{equation}
The derivative checks are
\begin{equation}
q_{0.05}'(1)=-0.339124730022,
\qquad
q_{0.05,h}^{\mathrm{FD}\prime}(1)=-0.339124730075
\quad(h=10^{-5}).
\label{eq:supp-derivative-checks}
\end{equation}
The absolute discrepancy is $5.34\times10^{-11}$.  The analytic benchmark
is returned by the script's family-specific integral derivative routine; it
is not a smaller-step finite difference.

All Monte Carlo calculations use the following random-number-generator (RNG) configuration:
\begin{verbatim}
RNGkind("Mersenne-Twister", "Inversion", "Rejection")
\end{verbatim}
under R 4.5.2 in the supplied run.  The fixed lower-stratum and transition
limiting samples use $10^6$ replications and seed $41000+10c$ for
$c\in\{0,2,4\}$.  The rank-one top-stratum sample uses seed $9002$.  The
finite-sample Wishart experiments use $5\times10^5$ replications and seed
$30000+10n+10c$.  The nuisance simulations use $10^6$ replications and seed
$50000+n$.  The anisotropic lower-stratum rows use $10^6$ replications and
seed $61000+\operatorname{round}(100\gamma)$.  Quadrature, chi-bar-square
inversion, analytic differentiation, finite differences, and ascent are
deterministic once their tolerances are fixed.

For the lower-stratum Gaussian model in
Subsection~\ref{subsec:lower-stratum-finite} of the main article, if
$d_1\ge d_2\ge d_3$ are the sample-covariance eigenvalues, spectral
profiling gives
\begin{equation}
\Lambda_n
=
n\sum_{j=2}^3
\{d_j-1-\log d_j\}\mathbf1_{\{d_j>1\}}.
\end{equation}
Under $\Sigma_n=\diag(c/\sqrt n,0,0)$, the local mean in whitened
coordinates is $cE_{11}/\sqrt2$, where $E_{11}$ is the first diagonal matrix
unit, yielding the reported transition law.  The
implementation uses the Bartlett decomposition and avoids storing the
underlying observations.

For the anisotropic lower-stratum experiment, the bisection equation
\eqref{eq:supp-lower-aniso-tau} is used for generic points with 70 bisection
steps; the closed-endpoint formulas
\eqref{eq:supp-lower-aniso-endpoint-small-gamma}--
\eqref{eq:supp-lower-aniso-endpoint-large-gamma} handle exceptional axes.
The simulated
zero masses $0.800194$, $0.853997$, and $0.911943$ at
$\gamma=0.25,1,4$ agree with the geometric values
$1-v_3(\gamma)=0.799751,0.853553,0.912211$.  The corresponding upper
$5\%$ critical values are $0.444867$, $0.227754$, and $0.035953$.

The script also checks nonnegativity of all simulated likelihood-ratio
statistics, unit sums of the chi-bar-square weights, monotonicity of the
selected anisotropic critical values, and the two coupled pathwise bounds.
It writes \texttt{sessionInfo.txt} and a human-readable \texttt{README.txt}
automatically.

The Armijo run resets the trial step to one at every iteration, halves
rejected steps, uses constant $10^{-4}$, gradient tolerance $10^{-7}$, and
minimum step $10^{-12}$.  When the minimum step is reached, the script writes
the status \texttt{step floor}, exactly as displayed in Table~\ref{tab:ascent}.  The summary
table in the main article is read directly from \texttt{ascent.csv}; the
detailed iteration audit is \texttt{ascent\_trace.csv}.

\end{document}